\documentclass[11pt,draftcls,onecolumn]{IEEEtran}

\usepackage{amsfonts,amsmath,amssymb,amsthm,bm}
\usepackage{cite}
\usepackage{color,cases,float}
\usepackage{graphicx}
\usepackage{hyperref}
\usepackage{subfigure}
\usepackage{url}
\usepackage[dvipsnames]{xcolor}
\usepackage[all,cmtip]{xy}

\definecolor{darkblue}{rgb}{0.0,0.0,0.6}
\hypersetup{colorlinks,breaklinks,
	linkcolor=darkblue,urlcolor=darkblue,
	anchorcolor=darkblue,citecolor=darkblue}

\newtheorem{corollary}{Corollary}
\newtheorem{proposition}{Proposition}
\newtheorem{assumption}{Assumption}
\newtheorem{theorem}{Theorem}

\newtheorem{lemma}{Lemma}
\newtheorem{remark}{Remark}

\def\lhs#1{\text{LHS}#1}
\def\rhs#1{\text{RHS}#1}
\def\cred #1{{\color{black} #1}}
\def\an#1{{\color{black}#1}}
\def\ws#1{{\color{black}#1}}

\def\argmin{\mathop{\rm argmin}}

\def\0{{\bf 0}}
\def\1{{\bf 1}}

\def\argmin{\mathop{\rm argmin}}

\def\beq{\begin{equation*}}
	\def\eeq{\end{equation*}}
\def\bql{\begin{equation}}
	\def\eql{\end{equation}}
\def\bqn{\begin{eqnarray*}}
	\def\eqn{\end{eqnarray*}}
\def\bnl{\begin{eqnarray}}
	\def\enl{\end{eqnarray}}
\def\bma{\begin{bmatrix}}
	\def\ema{\end{bmatrix}}
\def\bmx{\begin{matrix}}
	\def\emx{\end{matrix}}
\def\ben{\begin{enumerate}}
	\def\een{\end{enumerate}}
\def\bit{\begin{itemize}}
	\def\eit{\end{itemize}}
\def\bei{\begin{itemize}}
	\def\eei{\end{itemize}}
\def\bet{\begin{tabular}}
	\def\eet{\end{tabular}}

\newcommand{\ba}{\mathbf{a}}
\newcommand{\bA}{\mathbf{A}}
\newcommand{\bB}{\mathbf{B}}
\newcommand{\bb}{\mathbf{b}}
\newcommand{\bc}{\mathbf{c}}
\newcommand{\bd}{\mathbf{d}}

\newcommand{\bJ}{\mathbf{J}}
\newcommand{\dbJ}{\nabla\mathbf{J}}
\newcommand{\R}{\mathbb{R}}

\newcommand{\GTVdir}{\mathcal{G}_\mathrm{TV}^\mathrm{dir}}
\newcommand{\G}{\mathcal{G}}
\newcommand{\V}{\mathcal{V}}
\newcommand{\E}{\mathcal{E}}
\newcommand{\A}{\mathcal{A}}
\newcommand{\bx}{\mathbf{x}}
\newcommand{\by}{\mathbf{y}}
\newcommand{\tby}{\tilde{\mathbf{y}}}

\newcommand{\bq}{\mathbf{q}}
\newcommand{\br}{\mathbf{r}}
\newcommand{\bz}{\mathbf{z}}
\newcommand{\bs}{\mathbf{s}}
\newcommand{\f}{\mathbf{f}}
\newcommand{\g}{\mathbf{g}}

\newcommand{\df}{\nabla\mathbf{f}}
\newcommand{\sdf}{\widetilde{\nabla}\mathbf{f}}
\newcommand{\sdfi}{\widetilde{\nabla}f_i}

\newcommand{\dfj}{\nabla f_j}

\newcommand{\sdg}{\widetilde{\nabla}\mathbf{g}}

\newcommand{\bh}{\mathbf{h}}
\newcommand{\sdh}{\widetilde{\nabla}\mathbf{h}}
\newcommand{\dbh}{\nabla\mathbf{h}}

\newcommand{\BW}{\widetilde{W}}

\newcommand{\one}{\mathbf{1}}
\newcommand{\Fro}{\mathrm{F}}

\newcommand{\T}{\top}

\newcommand{\spa}[1]{\mathrm{span}\{#1\}}
\newcommand{\nul}[1]{\mathrm{null}\{#1\}}

\newcommand{\Wco}{\mathbf{C}}

\newcommand{\wco}{C}
\newcommand{\wdo}{W}
\newcommand{\Ni}{\mathcal{N}_i}
\newcommand{\Nj}{\mathcal{N}_j}
\newcommand{\Niin}{\mathcal{N}_i^{\mathrm{in}}}

\newcommand{\st}{\text{s.t.}}
\newcommand{\sigmax}[1]{\sigma_{\max}\left\{#1\right\}}

\newcommand{\bu}{\mathbf{u}}

\newcommand{\bv}{\mathbf{v}}
\newcommand{\bV}{\mathbf{V}}

\newcommand{\dia}[1]{\mathrm{diag}\left\{#1\right\}}

\newcommand{\zero}{\mathbf{0}}
\newcommand{\Omg}{{\bm{\Omega}}}
\newcommand{\met}{\mathcal{M}}
\newcommand{\D}{\Delta}
\newcommand{\lams}[1]{\tilde{\lambda}_{\min}\{#1\}}
\newcommand{\laml}[1]{\lambda_{\max}\{#1\}}
\newcommand{\sigl}[1]{\sigma_{\max}\{#1\}}
\newcommand{\proj}{\mathcal{P}}
\newcommand{\Lag}{\mathcal{L}}

\newcommand{\mapA}{\mathcal{A}}
\newcommand{\mapB}{\mathcal{B}}
\newcommand{\bX}{\mathcal{X}}

\newcommand{\tL}{L'}
\newcommand{\sdhi}{\widetilde{\nabla}h}
\newcommand{\bxi}{\boldsymbol{\xi}}
\newcommand{\Lap}{\textbf{{\L}}}
\newcommand{\Laps}{{\textbf{\L}}_{\G}}
\newcommand{\eLap}{\textbf{\c{L}}}

\DeclareMathOperator*{\Min}{minimize}
\DeclareMathOperator*{\ST}{subject\ to}

\begin{document}

\title{\LARGE{Improved Convergence Rates for Distributed Resource Allocation}}
\author{Angelia Nedi\'c, Alex Olshevsky, and Wei Shi
\thanks{A. Nedi\'c and W. Shi are with the School of Electrical, Computer and Energy Engineering at the Arizona State University, Tempe, AZ 85281, United States. A. Olshevsky is with the Department of Electrical and Computer Engineering at the Boston University, Boston, MA 02215, United States. Corresponding author: Wei Shi. Email:
\href{mailto:Wilbur.Shi@asu.edu}{Wilbur.Shi@asu.edu}. The work has been supported by the  Office of Naval Research under grant number N000014-16-1-2245 and the Air Force under grant number AF FA95501510394.}}

\maketitle

\begin{abstract}
In this paper, we develop a class of decentralized algorithms for solving a convex resource allocation problem \an{in a network of $n$ agents, where the agent objectives are decoupled while the resource constraints are coupled.} \an{The agents communicate over a connected undirected graph, and they want to collaboratively determine a solution to the overall network problem, while each agent only communicates with its neighbors.} We first study the connection between the decentralized resource allocation problem and the decentralized consensus optimization problem. Then, \an{using a class of algorithms for solving consensus optimization problems}, we propose a novel class of decentralized schemes for solving resource allocation problems \an{in a distributed manner}. Specifically, we first \an{propose} an algorithm for solving the resource allocation problem with an $o(1/k)$ convergence rate guarantee when the agents' objective functions are generally convex (could be nondifferentiable) and per agent local convex constraints are allowed; We then propose a gradient-based algorithm for solving the resource allocation problem when per agent local constraints are absent and show that such scheme can achieve geometric rate when the objective functions are strongly convex and have Lipschitz continuous gradients. We have also provided scalability/network dependency analysis. Based on these two algorithms, we have further proposed a gradient projection-based algorithm which can handle smooth objective and simple constraints more efficiently. Numerical experiments demonstrates the viability and performance of all the proposed algorithms.
\end{abstract}

\IEEEpeerreviewmaketitle
\newpage
\section{Introduction}
\an{\IEEEPARstart{T}his paper deals with a decentralized resource allocation problem, which is defined over a connected network of $n$ agents, as follows:}
\ws{
	\begin{subequations}\label{eq:basic}
		\begin{align}
			&\an{\Min\limits_{\bx=(x_1^\T;\ldots; x_n^\T)\in\R^{n\times p}}~\f(\bx)\triangleq }\sum\limits_{i=1}^n f_i(x_i)\label{eq:basic_line1} \\
			&\ST~\sum\limits_{i=1}^n (x_i-r_i)=0,\label{eq:basic_line2}\\
			&\qquad\qquad~~x_j\in\Omega_j,\quad \Omega_j\subseteq\R^{p}, \qquad\forall j=1,\ldots,n.\label{eq:basic_line3}
		\end{align}
	\end{subequations}
}
\an{
	For each agent $i$, the vector $x_i\in\R^p$ is its local decision variable. The objective 
	function $f_i: \R^p\rightarrow \R$ is convex and the constraint set $\Omega_i\subseteq\R^p$ is a nonempty closed and convex set, both of which are privately known by agent $i$ only. The equality constraints, $\sum_{i=1}^n (x_i-r_i)=0$, are coupling the agents' decisions, where $r_i\in\R^p$
	is a given resource demand vector for agent $i$.} 
\subsection{Literature review}\label{sec:Lit}
\an{A particular problem that falls under the preceding resource allocation formulation is 
	the economic dispatch problem in which each $f_i$ is a quadratic function and 
	every constraint set $\Omega_i$ is a box,} when the direct current power flow model is used \cite{Seifi2011electric}. Problems sharing similar forms have received extensive attention 
\an{due to the emergence of smart city concepts.} For example, 
references \cite{Zhang2014efficient} and \cite {Guo2016distributed} both consider the economic dispatch in 
a smart grid with an extra consideration of a random wind power injection. 
Algorithms proposed in both references are accompanied with discussions of basic convergence properties. \ws{Some earlier theoretical papers which have focused on decentralized algorithm design for solving the ``unconstrained version'' ($\Omega_j=\R^p,\ \forall j$) of \eqref{eq:basic} are available in the literature \cite{necoara2013random,lakshmanan2008decentralized}. Reference \cite{necoara2013random} considers a class of algorithms that randomly pick pairs of neighbors to perform updates. Under convexity assumption, an $O(L/(k\lambda_2))$ rate on the objective optimality residual in expectation is derived over fixed graphs; under strong convexity assumption, an $O\left((1-\kappa_{\f}^{-1}\lambda_2)^k\right)$ geometric rate is obtained also on the expectation of the objective optimality residual. Here, $k$ is the number of iterations the concerned algorithm has performed, \an{and $L$ is the gradient Lipschitz constant for the objective function $\f$.} The quantity $\kappa_{\f}$ is the condition number of the function $\f$ which is a scalar (no less than $1$) defined as the ratio of the gradient Lipschitz constant $L$ and the strong convexity constant $\mu$ of $\f$. The quantity $\lambda_2$ is the second smallest eigenvalue of a certain graph-dependent matrix. With a uniform assignment of probabilities, $\lambda_2^{-1}$ scales at the order of $O(n^4)$ (though it is possible to considerably improve on this if the probabilities are chosen in a centralized way depending on the graph). Reference \cite{lakshmanan2008decentralized} gives an algorithm which is shown to have an $O(LBn^3/k)$ rate for the decay of the squared gradient consensus violation over time-varying graph sequences; here $B$ is a constant which measures how long it takes for a time-varying graph sequence to be jointly connected.} \an{Reference \cite{Kar2012distributed} proposes a ``consensus plus innovations'' 
	method for solving problem~\eqref{eq:basic}, and the convergence of the method is established
	for quadratic objectives $f_i$ under a diminishing step size selection. Based on the alternating direction method of multipliers (ADMM), reference~\cite{Chang2015multi} provides a class of algorithms which can handle problem \eqref{eq:basic} with convergence guarantees.}  \an{In particular, under the assumption that the objective functions are convex, the convergence properties are established; when the per-agent constraints \eqref{eq:basic_line3} are absent 
	(i.e., $\Omega_j=\R^p, \forall j$), under the assumptions that the objective functions are strongly convex and have Lipschitz continuous gradients, a linear convergence (geometric) rate is shown.} By using the ADMM, if a center (in a star-shaped network) is allowed to carry a part of computational tasks, a more general problem formulation beyond \eqref{eq:basic} can be handled. Such a formulation and its distributed algorithms have been found to be useful in Internet services over hybrid edge-cloud networks \cite{Huang2017collaborative}. Reference \cite{doan2017distributed} studies the special case when $\Omega_j=\R^p, \forall j$, and considers solving the problem over time-varying networks. 
	Under the  strong convexity and the gradient Lipschitz continuity of the objective function $\f$,
	the algorithm in reference~\cite{doan2017distributed} is proved to have a geometric convergence rate $O\left((1-\kappa_{\f}^{-1}n^{-2})^k\right)$. In other words, for 
	the algorithm in~\cite{doan2017distributed} to reach an $\varepsilon$-accuracy, the number of iterations needs to be of the order $O\left(\kappa_{\f}n^2\ln(\varepsilon^{-1})\right)$. 
	This translates to an $O(\kappa_{\f}n^2)$ scalability in the number $n$ of agents, and it is the best scalability result (with the size $n$ of the network) that currently exists in the literature.
	Reference~\cite{Doan2016distributed} proposes a dual-based algorithm with a diminishing step size for solving \eqref{eq:basic}, for which an $O(1/\sqrt{k})$ convergence rate is derived. 
	However, such algorithms with vanishing step sizes cannot be extended to handle problems with time-varying objectives and usually exhibits poor convergence. 
	\cred{A recent work~\cite{Aybat2016distributed} has proposed a class of algorithms to handle 
	the resource sharing problem under the conic constraints. The algorithms are built on a modified Lagrangian function and an ADMM-like scheme for seeking a saddle point of the Lagrangian function, 
	which has been shown to have an ergodic $O(1/k)$ rate for agents' objective functions and constraints violation. The problem formulation in \cite{Aybat2016distributed} treats \eqref{eq:basic} as a special case: the constraint \eqref{eq:basic_line2} is replaced by the more general $\sum\limits_{i=1}^n (R_ix_i-r_i)\in\mathcal{K}$ where $\mathcal{K}$ is a convex cone and $R_i$ is a matrix that couples local resources. It was required that the interior of $\mathcal{K}$ is nonempty which does not apply to \eqref{eq:basic_line2}. The authors recently removed such nonempty interior requirement on $\mathcal{K}$ during our preparation of this paper. 
	We also would like to point out that our rates are non-ergodic and the measures/criteria used for our rates have some advantages in practical uses (see the comments following Theorem \ref{theorem:o_1_k}).} 
	Based on consensus and push-sum approaches~\cite{Nedic2013}, a recent reference \cite{yang2016distributed} 
	proposes a distributed algorithm for solving problem~\eqref{eq:basic} 
over time-varying directed networks and provides convergence guarantees. 
\an{Aside from the above algorithms which are all discrete-time methods, 
	there are some continuous-time algorithms such as the one in reference~\cite{cherukuri2016initialization},
	where convergence under general convexity assumption is ensured.}
\an{Table \ref{tab: sumup} summarizes the most relevant references with the convergence rates and the scalability results for distributed algorithms for solving problem~\eqref{eq:basic}, and it illustrates the results of this paper with respect to the existing work.} A very recent work by \cite{Scaman2017optimal} proposes algorithms for decentralized consensus optimization for smooth and strongly convex objectives. By applying Nesterov's acceleration to the dual problem of the consensus optimization, the algorithms in \cite{Scaman2017optimal} attain optimal geometric convergence rate of the first-order algorithms, and that a decentralized resource allocation algorithm can scale in the order of $O(\sqrt{\kappa_{\f}}n)$ with the number $n$ of agents. Nevertheless, to enjoy this rate/scalability improvement, one needs to know the strong convexity constant $\mu$ and the gradient Lipschitz constant $L$. In contrast, the algorithms we study in this paper only ask for knowing the  parameter $L$. Furthermore, the algorithms and analysis in \cite{Scaman2017optimal} are specified for smooth strongly convex objectives. It is unclear how one can modify such schemes in order to solve convex problems, nonsmooth problems, or problems with constraints/projections. 

\begin{table}[H]
	\centering\caption{The convergence rates and scalability results for distributed resource allocation algorithms \an{for problem~\eqref{eq:basic}, which is convex in all instances.} The scalar $L$ is the Lipschitz-gradient constant, while the condition number $\kappa_\f=L/\mu$ where $\mu$ is the strong convexity constant for $\f$. The rates are given in terms of the number $k$ of iterations, while the ``scalability" column shows how the algorithm's geometric rate depends on the number of agents, $n$, and the condition number, $\kappa_\f$. By saying ``unconstrained'' in the table, 
		we mean that $\Omega_i=\R^p,\ \forall i$. The quantity $\lambda_2$ used in reference \cite{necoara2013random} is the second smallest eigenvalue of a certain graph-dependent matrix (see Subsection \ref{sec:Lit} of this paper for more details).\label{tab: sumup}}
	\smallskip
	\begin{tabular}{|c|c|c|c|c|}
		\hline\hline
		Reference & uncon. strongly convex& unconstrained & constrained &  scalability\\
		\hline\hline
		\cite{necoara2013random} & geometric & $O(1/k)$ & -- &  $O(\kappa_{\f}\lambda_2^{-1})$\\
		\hline
		\cite{lakshmanan2008decentralized} & -- & $O(1/k)$ & -- & --\\
		\hline 
		\cite{Chang2015multi} & geometric & -- & -- &   --\\
		\hline
		\cite{doan2017distributed} & geometric & -- & $O(1/k)$ & $O(\kappa_{\f}n^2)$\\		
		\hline
		\cite{Doan2016distributed} & -- & $O(1/\sqrt{k})$ & $O(1/\sqrt{k})$ & -- \\		  
		\hline
		\cite{Aybat2016distributed} & geometric & $O(1/k)$ & $O(1/k)$ &-- \\
		\hline
		this paper & geometric & $o(1/k)$ & $o(1/k)$ &$O(n^2+\sqrt{\kappa_{\f}}n)$\\	
		\hline
	\end{tabular}
\end{table}

\subsection{Our contributions}
In this paper, we design an algorithm for solving problem~\eqref{eq:basic} from an unconventional point of view. 
We consider synchronous updates and connected undirected communication networks with time-varying topologies. 
For general convex functions $f_i$ (without requirements of strong convexity and smoothness), our basic 
method has $o(1/k)$ convergence rate\footnote{A nonnegative sequence $\{a_k\}$ is said to be convergent to $0$ at an $O(1/k)$ rate if ${\lim\sup}_{k\rightarrow\infty} ka_k<+\infty$. In contrast, it is said to have an $o(1/k)$ rate if ${\lim\sup}_{k\rightarrow\infty} ka_k=0$.}, 
which is slightly better than the sub-linear convergence rates achieved in the literature. 
\an{When the objective functions $f_i$ are strongly convex and smooth, 
	and $\Omega_i=\R^p$ for all $i$, 
	we show a geometric convergence of the method. 
	Furthermore, we  
	find that the algorithm scales in the order of $O(n^2+\sqrt{\kappa_{\f}}n)$, with the number $n$ of agents, which is better than the best scaling that has been currently achieved in the literature 
	(see reference~\cite{doan2017distributed} where the scaling is $O(\kappa_{\f}n^2)$).
	For the case when the objective function is smooth and $\Omega_i=\R^p$ for all $i$, we also 
	provide a gradient-based algorithm that achieves an $o(1/k)$ rate under convexity assumption, 
	and a geometric rate under the strong convexity assumption for the objective functions $f_i$. 
	Finally, based on these two methods, we provide a combined optimization strategy which finds an 
	optimal solution of problem~\eqref{eq:basic} by using a gradient-projection at each iteration.}

\section{Resource Allocation and Its Connection to Consensus Optimization}
\subsection{Notation and basic assumptions}\label{sec:notation}
Some of the notation may not be standard 
but it enables us to present our algorithm and analysis in a compact form.
Throughout the paper, we let agent $i$ hold a local variable $x_i$, a function $f_i$, and a constraint set 
$\Omega_i$ of problem~\eqref{eq:basic}.
We define 
\[\Omg\triangleq\Omega_1\times\cdots\times\Omega_n.\] 

\an{Our basic assumption is that problem~\eqref{eq:basic} is convex, which is formalized as follows.}
\begin{assumption}(Functional properties)\label{assum:all}
	\an{For any $i\in\{1,2,\ldots,n\}$, the function $f_i:\R^p\to\R$ is convex} while the set $\Omega_i\subseteq\R^p$ is nonempty, closed and convex.
\end{assumption}

We define $g_i$ as the indicator function of the set $\Omega_i$, namely,
\[
g_i(x_i)=\left\{
\begin{array}{ll}
0,&\text{ if $x_i\in\Omega_i$,}\\
+\infty,&\text{ if $x_i\notin\Omega_i$.}	
\end{array}
\right.
\]
We also define a composite function $h_i$ for agent $i$, as follows:
	\[h_i\triangleq f_i+g_i:\R^p\rightarrow\R\cup\{+\infty\},\quad\forall i=1,\ldots,n.\] 
	Under Assumption~\ref{assum:all}, the functions $g_i:\R^p\to\R\cup\{+\infty\}$ are proper, closed, and convex,
	and so are the functions $h_i=f_i+g_i$ since the domain of $f_i$ is $\R^p$.
	Furthermore, under Assumption~\ref{assum:all}, the subdifferential sets $\partial h_i(x_i)$ satisfy (see Theorem 23.8 of~\cite{Rock1970})
	\begin{equation}\label{eq:subdif}
		\partial h_i(x_i)=\partial f_i(x_i)+\partial g_i(x_i)\qquad\hbox{for all $x_i\in\R^p$.}\end{equation}
	 The equality in \eqref{eq:subdif} holds when $\text{ri}(\text{dom} \{f_i\})\bigcap\text{ri}(\text{dom} \{g_i\})\neq\emptyset$, where $\text{ri}(\cdot)$ denotes the relative interior of a set and 
	 $\text{dom}\{\cdot\}$ is the (effective) domain of a function (see Section 4 of~\cite{Rock1970} for the definition of ``(effective) domain''). See also Remark 16.46 and Corollary 16.48 of \cite{Bauschke2011convex} for more conditions and comments for \eqref{eq:subdif} to hold. Note that we have not imposed any differentiability on $f_i$'s. 
	 Moreover, since $\partial g_i(x_i)$ coincides with the normal cone 
	of $\Omega_i$ at $x_i\in\Omega_i$, we have that
	\[\partial g_i(x_i)\ne \emptyset\qquad\hbox{for all $x_i\in\Omega_i$}.\]

In addition to the network objective $\f(\bx)$ defined in~\eqref{eq:basic_line1}, we introduce two more network-wide aggregate functions,
\begin{equation}\label{eq:fx_and_gy}
	\g(\bx)\triangleq\sum\limits_{i=1}^{n} g_i(x_i),\qquad\ \text{and}\qquad 
	\bh(\bx)\triangleq\sum\limits_{i=1}^{n} h_i(x_i),
\end{equation}
where
\begin{equation}\label{eq:x}
	\bx\triangleq\left(
	\begin{array}{ccc}
		\textrm{---}& x_1^\T & \textrm{---} \\
		\textrm{---}& x_2^\T & \textrm{---} \\
		&\vdots& \\
		\textrm{---}& x_n^\T & \textrm{---} \\
	\end{array}
	\right)\in\R^{n\times p}.
\end{equation}
\ws{Similarly, we define a matrix $\br$ by using the vectors $r_i$, $i=1,\ldots,n$.}

\an{Letting $\widetilde{\nabla} f_i(x_i)$ be a subgradient of $f_i$ at $x_i$,
	we construct a matrix $\sdf(\bx)$ of subgradients $\widetilde{\nabla} f_i(x_i)$, as follows:}
\begin{equation}\label{eq:dfx}
	\sdf(\bx)\triangleq\left(
	\begin{array}{ccc}
		\textrm{---}& (\widetilde{\nabla} f_1(x_1))^\T & \textrm{---} \\
		\textrm{---}& (\widetilde{\nabla} f_2(x_2))^\T & \textrm{---} \\
		&\vdots & \\
		\textrm{---}& (\widetilde{\nabla} f_n(x_n))^\T & \textrm{---} \\
	\end{array}
	\right)\in\R^{n\times p},
\end{equation}
and, similarly, the matrices $\sdg(\bx)$ and $\sdh(\bx)$ are defined using 
subgradients of $g_i$ and $h_i=f_i+g_i$ at $x_i$, respectively. 
We drop the tilde in the notation $\widetilde\nabla$ when the function under consideration 
	is differentiable (i.e., a subdifferential set contains only a gradient).
Each row $i$ of $\bx$, $\br$, $\sdf(\bx)$, $\sdg(\bx)$, and $\sdh(\bx)$ corresponds to the information available 
to agent $i$ only. 

\an{We use $\1$ to denote a vector with all entries equal to 1, where the size of the vector is to be understood from the context.
	We say that a matrix $A\in\R^{n\times p}$ is \emph{consensual} if its rows are identical, i.e., 
	$A_{i:}=A_{j:}$ for all agents $i$ and $j$.}
For a given matrix $A$, $\|A\|_\Fro$ stands for its Frobenius norm, while $\sigmax{A}$ stands for its spectral norm (largest singular value). The largest eigenvalue of 
a symmetric positive semidefinite matrix $A$ is denoted by $\laml{A}$, while its
	smallest non-zero eigenvalue is denoted by $\lams{A}$. For any matrix $A\in\R^{m\times n}$, the set $\nul{A}\triangleq\{x\in\R^n\mid Ax=0\}$ is the null space of $A$, \an{while} 
the set $\spa{A}\triangleq \{y\in\R^m\mid y=Ax,\forall x\in\R^n\}$ is the linear span of all the columns of $A$. Given a matrix $A\in\R^{m\times p}$ and a matrix $B\in\R^{m\times p}$, the standard inner product of them is represented as $\langle A,B\rangle=\text{Trace}\{A^\T B\}$. Given a positive (semi)definite square matrix $\met\in\R^{m\times m}$, 
we define the $\met$-weighted (semi-)norm $\|A\|_{\met}=\sqrt{\langle A,\met A\rangle}$. In the analysis, we use the abbreviation $\lhs{i}$ to refer to the $i$-th summand in the left-hand-side of an (in)equality; likewise, $\rhs{i}$ is used to refer to the $i$-th summand in the right-hand-side of an (in)equality. 

To model the underlying communication network for the agents, we use a simple (no self-loop) undirected graph, where $[n]=\{1,2,\ldots,n\}$ is the vertex set 
and $\E$ is the edge set. \ws{We say that an $n\times n$ matrix $A$ is compatible with the graph $\G$ when the following property holds: $\forall i,j\in[n]$, the $(i,j)$-th entry of $A$ is zero if neither \an{$\{i,j\}$} is an element of $\E$ nor $i\neq j$.} \an{We use ${\cal N}_i$ to denote the set of neighbors of agent $i$ in the graph $\G$,
	i.e., ${\cal N}_i=\{j\in[n]\mid \{i,j\}\in \E\}.$}

Let $\Laps$ denote the (standard) Laplacian matrix associated with the graph $\G$,
i.e., $\Laps=D-J$, where $D$ is the diagonal matrix with diagonal entries $D_{ii}=d_i$ and $d_i$ 
being the number of edges incident to node $i$, 
while $J$ is the graph adjacency matrix (with $J_{ij}=1$ when $\{i,j\}\in \E$ and $J_{ij}=0$ otherwise). A few facts about $\Laps$ are that $\Laps$ is compatible with $\G$, symmetric and positive semidefinite.

In our algorithm, we will use a matrix $\Lap=[\text{\L}_{ij}]$ whose behavior is ``close or the same'' to $\Laps$, in the sense of the following assumption. 

\begin{assumption}[Graph connectivity and null/span property]\label{ass:conn_and_prop}
	The graph $\G$ is connected and a matrix $\Lap$ is compatible with the graph $\G$. 
	\ws{Furthermore, $\Lap=U^\T U$ for some full row-rank matrix $U\in\R^{(n-1)\times n}$} and 
	\an{$\nul{U}=\{a\one\in\R^n \mid a\in\R\}$.}
\end{assumption}
Note that Assumption \ref{ass:conn_and_prop} equivalently says that $\nul{U}=\nul{\Lap}$ and thus $x_1=x_2=\cdots=x_n\Leftrightarrow\Lap\bx=0\Leftrightarrow U\bx=0$.
\an{The matrix $\Lap$ can be chosen in several different ways:
	\begin{itemize}
		\item[(i)] Since the graph Laplacian $\Laps$ satisfies Assumption~\ref{ass:conn_and_prop}, we can choose $\Lap=\Laps$. In this case, each agent needs to know the number of its neighbors (its degree) and $\Lap$ can be constructed without any communication among the agents.
		\item[(ii)] We can let $\Lap=\Laps/\laml{\Laps}$. The network needs $\laml{\Laps}$ to configure this matrix but a preprocessing to retrieve $\laml{\Laps}$ is possible \cite{tran2014distributed}.
		\item[(iii)] We can also choose $\Lap=0.5(I-W)$ where $W$ is a symmetric doubly stochastic matrix that is compatible with the graph $\G$ and $\laml{W-\one\one^\T/n}$ is strictly less than $1$. 
		This matrix can be constructed in the network through a few rounds of local interactions between the agents since some local strategies for determining $W$ exist, such as the Metropolis-Hasting rule which requires only one round of local interactions \cite{xiao2006,bacsar2016convergence}.
	\end{itemize}
}

\an{We will discuss the specific choices of $\Lap$ in some of our results to simplify analysis or to point out to interesting results.}

\subsection{The resource allocation and consensus optimization problems}\label{ss:resalprob}
In this subsection, we investigate the first-order optimality conditions for problem~\eqref{eq:basic} and for consensus optimization.
With the notation introduced in the preceding section, 
the resource allocation problem~\eqref{eq:basic} can be compactly given by	
\begin{equation}\label{eq:F}
	\begin{array}{cl}
		\min\limits_{\bx\in\R^{n\times p}} & \bh(\bx)=\sum\limits_{i=1}^n h_i(x_i),\\
		\st & \one^\T(\bx-\br)=\zero.
	\end{array}
\end{equation}
\an{where
	$\one$ is a vector of appropriate dimension whose entries are all equal to $1$.
	By using the Lagrangian function, 
	we can write down the optimality conditions for problem~\eqref{eq:F} in a special form, as given in the
	following lemma.}
\begin{lemma}[First-order optimality condition for \eqref{eq:F}]
	\label{lem:opc-res}
	\an{Let Assumptions \ref{assum:all} and \ref{ass:conn_and_prop}} hold, and let $c\ne0$ be a given scalar.
	Then,  
	$\bx^*$ is an optimal solution of \eqref{eq:F} 
	if and only if there exists a matrix \an{$\bq^*\in\R^{(n-1)\times p}$} 
	such that the pair $(\bx^*,\bq^*)$ satisfies the following relations:
	\begin{subequations}\label{eq:opt_cond}
		\begin{align}
			&\bx^*-\br+c U^\T \bq^*=\zero,\label{eq:oc_line1} \\
			&U\sdh(\bx^*)=\zero.\label{eq:oc_line2}
		\end{align}
	\end{subequations}
	where $U$ is the matrix defined in Assumption~\ref{ass:conn_and_prop}.
\end{lemma}

\begin{proof}
	The Lagrangian function of problem \eqref{eq:F} is
	$$
	\Lag(\bx,\by)=\bh(\bx)+\langle \by,(\bx-\br)^\T\one\rangle,
	$$
	where $\by\in\R^p$ contains all the multipliers of the constraints $(\bx-\br)^\T\one=\zero$. 
	\an{Since the problem is convex, 
		the necessary and sufficient optimality condition for the primal-dual pair $(\bx^*,\by^*)$} is
	\begin{subequations}
		\begin{align*}
			&\one^\T(\bx^*-\br)=\zero,\\
			&\zero\in\partial\bh(\bx^*)+\one(\by^*)^\T,
		\end{align*}
	\end{subequations}
	where $\partial\bh(\bx^*)$ is to be understood as a collection of all matrices
	whose every row $i$ is given by \an{some subgradient $(\widetilde\nabla h_i(x_i^*))^\T$ of $h_i(x_i)$ at $x_i=x_i^*$. 
		Noting that the condition $\zero\in\partial\bh(\bx^*)+\one(\by^*)^\T$ is equivalent to 
		the requirement that there exists an $\bx^*$ such that $\zero=\sdh(\bx^*)+ \one(\by^*)^\T$,
		the optimality condition for the primal-dual pair $(\bx^*,\by^*)$ can be written as:}
	\begin{subequations}\label{eq:lemma1_proof1}
		\begin{align}
			&\one^\T(\bx^*-\br)=\zero,\label{eq:lemma1_proof1_line2}\\
			&\sdh(\bx^*)= - \one(\by^*)^\T.\label{eq:lemma1_proof1_line1}	
		\end{align}
	\end{subequations}
	
		By Assumption~\ref{ass:conn_and_prop}, the full row-rank matrix $U\in\R^{(n-1)\times n}$ has the set 
		$\{a\one\mid\forall a\in\R\}$ as its null space, implying that $U \bq=0$ if and only if $\bq=a\one$ for some scalar $a$.
	Applying this result to each column of \eqref{eq:lemma1_proof1_line1}, we obtain its equivalent relation \eqref{eq:oc_line2}.
	Relation~\eqref{eq:lemma1_proof1_line2} states that the vector $\one$ is orthogonal to every column 
	of the matrix $\bx^*-\br$. The same is true if we scale $\bx^*-\br$ by a factor  $\frac{1}{c}$, 
	\an{where $c\ne 0$ is the given scalar in the lemma.}
	Since $\one^{\perp}= (\nul{U})^{\perp}=\spa{U^\T}$,
	it implies that $\one^\T\frac{1}{c}(\bx^*-\br)=\zero$ if and only if \an{every column of matrix 
		$\frac{1}{c}(\bx^*-\br)$ lies in the set $\spa{U^\T}$.
		Let $-q^*_j\in\R^{n-1}$ be a vector such that $U^\T(-q_j^*)$ is equal to the $j$-th column of 
		$\frac{1}{c}(\bx^*-\br)$, for $j=1,\ldots,p$. Take these column vectors as the columns of a matrix $\bq^*$, 
		for which after multiplying by $c$, we have $\bx^*-\br +cU^\T\bq^*=\zero$,  thus showing that relation~\eqref{eq:lemma1_proof1_line2} is equivalent to relation~\eqref{eq:oc_line1}.}
\end{proof}

\an{It turns out that the optimality conditions for the resource optimization problem, as given in
	Lemma~\ref{lem:opc-res}, have an interesting connection with the optimality conditions for the consensus optimization problem. In order to expose this relation, we next discuss the consensus optimization problem, 
	which is given as follows:} 
\begin{equation}\label{eq:F_consensus}
	\begin{array}{cl}
		\min\limits_{\bx\in\R^{n\times p}} & \bh(\bx)=\sum\limits_{i=1}^n
		h_i(x_i),\\
		\st & x_1=x_2=\cdots=x_n.
	\end{array}
\end{equation}
The local objective of each agent in \eqref{eq:F_consensus} is the same as that in \eqref{eq:F}. 
\an{Unlike the resource allocation problem, 
	instead of having $\sum_{i=1}^n (x_i-r_i)=0$ as constraints, 
	here we have the consensus constraints, i.e., $x_1=x_2=\cdots=x_n$.}

The first-order optimality condition of \eqref{eq:F_consensus} is stated in the following lemma. 
\begin{lemma}[First-order optimality condition for \eqref{eq:F_consensus}] 
	\label{lem:opc-con}
	\an{Let Assumptions~\ref{assum:all} and~\ref{ass:conn_and_prop} hold, and let $c\ne 0$ be a given scalar.}
	Then, $\bx^*$ is an optimal solution of \eqref{eq:F_consensus} if and only if there exists a matrix \an{$\bq^*\in\R^{(n-1)\times p}$} 
	such that the pair $(\bx^*,\bq^*)$ satisfies the following relations:
	\begin{subequations}\label{eq:opt_cond_consensus}
		\begin{align}
			&\sdh(\bx^*)+ c U^\T \bq^*=\zero,\label{eq:oc_line1_consensus}\\
			&U\bx^*=\zero,\label{eq:oc_line2_consensus}
		\end{align}
	\end{subequations}
	where $U$ is the matrix defined in Assumption~\ref{ass:conn_and_prop}.
\end{lemma}
The proof for this lemma is basically the same to that for Lemma 3.1 of reference \cite{Shi2015_2} only that we use the decomposition $\Lap=U^\T U$ while reference \cite{Shi2015_2} uses $U=\sqrt{\Lap}$.
\subsection{The mirror relationship}\label{sec:mirror}
\ws{It is known that the Lagrangian dual problem of the resource allocation problem is a consensus optimization problem (see the discussion around equations (4)$\sim$(6) in \cite{Chang2015multi}). As having been pointed out in reference \cite{Chang2015multi}, a distributed optimization method that can solve the consensus optimization problem may also be used for the resource allocation problem through solving the dual of the resource allocation problem. Here, we will provide more special relations these two problems have which leads to a class of resource allocation algorithms following the design of a class of decentralized consensus optimization algorithms. Due to such special relations, it is possible that one can give a decentralized resource allocation algorithm without investigating the Lagrangian dual relationship between the above mentioned two problems.}

\ws{The optimality conditions of the resource allocation problem \eqref{eq:F} and the consensus optimization problem \eqref{eq:F_consensus} are summed up in the following box:}
\begin{table}[H]
	\centering\caption{Summary of optimality conditions\label{tab: opt}}
	\begin{tabular}{c|c}
		\hline\hline
		Opt. Cond. of Resource Allocation,  \eqref{eq:opt_cond} & Opt. Cond. of Consensus Optimization, \eqref{eq:opt_cond_consensus}\\
		\hline\hline
		$\bx^*-\br=-c U^\T \bq^*$ & $\sdh(\bx^*)=-c U^\T \bq^*$\\	
		$U\sdh(\bx^*)=\zero$ & $U\bx^*=\zero$\\
		\hline
	\end{tabular}
\end{table}
\ws{These conditions share the same structure, i.e.,
	\begin{equation}\nonumber
		\begin{array}{c}
			\mapA(\bx^*)=-cU^\T\bq^*,\\
			U\mapB(\bx^*)=\zero,
		\end{array}
	\end{equation}
	where $\mapA:\R^{n\times p}\rightarrow\R^{n\times p}$ and $\mapB:\R^{n\times p}\rightarrow\R^{n\times p}$ are some general maps. The only difference between relations \eqref{eq:opt_cond} and \eqref{eq:opt_cond_consensus} 
	is: what in the span space of $U^\T$ is and what in the null space of $U$  is. 
	In the resource allocation problem, we need $\sdh(\bx)$ to be consensual while the rows of $\bx-\br$ summing to $0$. In the consensus optimization problem, we need $\bx$ to be consensual while the rows of $\sdh(\bx)$ are summing to $\zero$. If in \eqref{eq:opt_cond_consensus}, we replace $\sdh(\bx^*)$ by $\bx^*-\br$ and $\bx^*$ by $\sdh(\bx^*)$ at the same time, it will recover \eqref{eq:opt_cond}. Hypothetically, in an iterative consensus optimization algorithm with iteration index $k$, if we substitute the image of the ``subgradient map'' $\sdh(\bx^k)$ by that of some other map $\mapA(\bx^k)$ and substitute the image of the identity map $\bx^k$ by that of some other map $\mapB(\bx^k)$, there is a chance that when $k\rightarrow\infty$, $\mapA(\bx^k)$ still goes into the span space of $U^\T$ and $\mapB(\bx^k)$ still goes into the null space of $U$.}

\ws{Furthermore, to analyze a consensus convex optimization algorithm, the most crucial relation we need is the monotone inequality, namely,
	$0\leq\langle\sdh(\bx^k)-\sdh(\bx^*),\bx^k-\bx^*\rangle$ for all $k$, which translates into verifying $0\leq\langle\mapA(\bx^k)-\mapA(\bx^*),\mapB(\bx^k)-\mapB(\bx^*)\rangle$ for all $k$, when we substitute the key quantities by the images of those general maps we have discussed above. Apparently, this inequality still holds when we let $\mapA(\bx)=\bx-\br$ and $\mapB(\bx)=\sdh(\bx)$ and assume $\bh(\bx)$ is convex. It is possible that such substitutions of $\mapA$ and $\mapB$ do not affect the validity of some of the existing analyses for certain algorithms. 
	For example, the subgradient form of the proximal method is 
	$\bx^{k+1}=\bx^k-\alpha\sdh(\bx^{k+1})$ where $\alpha$ is a step size. This method can be proven to have $\sdh(\bx^{k+1})\rightarrow\zero$ if $\bh(\bx)$ is convex. Its counterpart after the substitution is $\sdh(\bx^{k+1})=\sdh(\bx^{k})-\alpha(\bx^{k+1}-\br)$, which can be resolved as
	$$\text{initializing with arbitrary }\bx^0\in\R^{n\times p}\quad\text{ and }\quad\bd^0=\sdh(\bx^{0}),
	$$
	and updating according to the following rules:
	$$
	\bx^{k+1}=\argmin_{\bx\in\R^{n\times p}} \left\{\bh(\bx)+\frac{\alpha}{2}\left\|\bx-\br-\frac{\bd^k}{\alpha}\right\|_\Fro^2\right\}
	\qquad\text{ and }\qquad\bd^{k+1}=\bd^{k}-\alpha(\bx^{k+1}-\br).
	$$
	It can be shown that the sequence $\{\bx^{k}\}$ of such an iterative algorithm will converge to $\br$, corresponding to the fact that $\{\sdh(\bx^k)\}$ will converge to $\zero$ in the proximal method.}

These observations motivate the class of algorithms that we propose 
for solving resource allocation problems based on some existing consensus optimization algorithms. 
The consensus optimization algorithms \an{that will be exploited in this paper 
	are simple, efficient, and have recently been further accelerated akin to Nesterov's fast methods \cite{zhao2015fast}, as well as enhanced to work over asynchronous \cite{wu2016decentralized} and directed communication networks \cite{Zeng2015,Xi2015}.}
Our algorithm design philosophy also implies possibilities of enhancing the resource allocation algorithms proposed in this paper by using techniques from those for advancing consensus optimization algorithms.

In Section~\ref{sec:algos}, we describe our resource allocation algorithms and provide 
their convergence analysis. 
Finally, we will illustrate some numerical experiments in Section~\ref{sec:numer}
and conclude the paper with remarks in Section~\ref{sec:concl}. 

\section{The Algorithms and their Convergence Analysis}\label{sec:algos}
Before we introduce our algorithms and conduct the analyses, 
\an{let us introduce 
	the solution set of the resource allocation problem \eqref{eq:basic}, denoted by $\bX^*$.
	We make the following assumption for problem~\eqref{eq:basic}, which we use throughout the paper.}

\an{\begin{assumption}\label{ass:sol}
		The solution set $\bX^*$ of the resource allocation problem~\eqref{eq:basic} is nonempty.
		Furthermore, a Slater condition is satisfied, i.e.,  
		there is a point $\tilde \bx$ which satisfies the linear constraints in~\eqref{eq:basic}
		and lies in the relative interior of the set constraint $\Omg=\Omega_1\times\cdots\times\Omega_n$.
	\end{assumption}
	The set $\bX^*$ is nonempty, for example, when the constraint set of the resource allocation problem~\eqref{eq:basic} is compact, or the objective function satisfies some growth condition. 
	Under the convexity conditions in Assumption~\ref{assum:all}, the optimal set $\bX^*$ is closed and convex.
	Under Assumptions~\ref{assum:all} and~\ref{ass:sol}, the strong duality holds for problem~\eqref{eq:basic} and its Lagrangian dual problem, and the dual optimal set is nonempty (see 
	Proposition 6.4.2 of \cite{Bertsekas2003}).}
\subsection{The basic algorithm: Mirror-P-EXTRA}
This algorithm solves the original problem~\eqref{eq:basic}, i.e., the resource allocation with local constraints. 
This basic algorithm operates as follows (Algorithm 1). 
\an{Each agent $i$ uses its local parameter $\beta_i>0$, 
	which can be viewed as the stepsize. 
}
\smallskip
\begin{center}
	{\textbf{Algorithm 1: Mirror-P-EXTRA}}
	
	\smallskip
	\begin{tabular}{l}
		\hline
		\emph{  } Each agent $i$ chooses its own parameter 
		\an{$\beta_i>0$ and the same parameter $c>0$;}\\
		\emph{  } Each agent $i$ initializes with $x_i^0\in\Omega_i$, $s_i^0=\sdfi(x_i^0)$, and $y_i^{-1}=0$;\\
		\emph{  } Each agent $i$ \textbf{for} $k=0,1,\ldots$ \textbf{do}\\
		\qquad$y_i^{k}=y_i^{k-1}+\sum_{j\in\Ni\cup\{i\}}\text{\L}_{ij}s_j^{k}$;\\
		\qquad
		$x_i^{k+1}=\argmin\limits_{x_i\in\Omega_i} \left\{
		f_i(x_i) - \langle s_i^k,x_i\rangle +\frac{1}{2\beta_i}\|x_i-r_i+2cy_i^k-cy_i^{k-1}\|_2^2\right\}$;\\
		\qquad$s_i^{k+1}=s_i^k-\frac{1}{\beta_i}\left(x_i^{k+1}-r_i+2cy_i^k-cy_i^{k-1}\right)$;\\
		\emph{  } \textbf{end}\\
		\hline
	\end{tabular}
\end{center}
\smallskip

The algorithm is motivated by the P-EXTRA algorithm from reference \cite{Shi2015_2} for 
a consensus optimization problem. The reason we refer to Algorithm 1 as 
\an{Mirror-P-EXTRA will be clear from the following lemma. In the lemma and later on, we will use
	$B$ to denote the diagonal matrix that has $\beta_i$ as its $(i,i)$-th entry,
	\[B=\left[
	\begin{array}{cccc}
	\beta_1 & 0 &\cdots &0\cr
	0& \beta_2 & \cdots & 0\cr
	\vdots&\vdots&\ddots& \vdots\cr
	0 & 0 & \cdots & \beta_n
	\end{array}\right].
	\]}

\begin{lemma}\label{lemma:rec_rel_alg1} 
	\an{Let Assumptions~\ref{assum:all} and~\ref{ass:conn_and_prop} be satisfied, and let $c>0$.}
	Then, the sequence $\{\bx^k,\by^k\}$ generated by Algorithm 1 satisfies for $k=0,1,\ldots$,
	\begin{subequations}\label{eq:updates2}
		\begin{align}
			&\bx^{k+1}-\br+cU^\T\bq^{k+1}+(B-c\Lap)(\sdh(\bx^{k+1})-\sdh(\bx^k))=\zero,\label{eq:up_line1}\\
			&\bq^{k+1}=\bq^k+U\sdh(\bx^{k+1}),\label{eq:up_line2}\\
			&\by^{k}=U^\T\bq^{k}\label{eq:up_line3},
		\end{align}
	\end{subequations}
	where $\bx^0$ is chosen the same as that in Algorithm 1, $\bq^{-1}=\zero$, and $\bq^0=U\sdh(\bx^0)$ where the matrix $\sdh(\bx^0)$ of subgradients is the same as those used in Algorithm 1.
\end{lemma}
\begin{proof}
	By using the notation given in Subsection~\ref{sec:notation}, the updates of Algorithm 1 can be represented compactly as initializing with arbitrary $\bx^0\in\Omg$, $\bs^0=\sdf(\bx^0)$, and $\by^{-1}=\zero$, and then performing for $k=0,1,2,\ldots$,
	\begin{subequations}\label{eq:updates2_proof1}
		\begin{align}
			&\by^k=\by^{k-1}+\Lap\bs^k,\label{eq:up_p1_line1}\\
			&\sdh(\bx^{k+1})-\bs^k+B^{-1}(\bx^{k+1}-\br+2c\by^k-c\by^{k-1})=\zero,\label{eq:up_p1_line2}\\
			&\bs^{k+1}=\bs^{k}-B^{-1}(\bx^{k+1}-\br+2c\by^k-c\by^{k-1}).\label{eq:up_p1_line3}
		\end{align}
	\end{subequations}
	\an{
		From \eqref{eq:up_p1_line2} and \eqref{eq:up_p1_line3}, also considering the initialization\footnote{Since 
			$\bx^0\in\Omg$, we can choose the subgradient $\sdg(\bx^0)=\zero$ and use  
			the relation $\sdf(\bx^0)+\sdg(\bx^0)=\sdh(\bx^0)$ (see~\eqref{eq:subdif}).}
		$\bs^0=\sdf(\bx^0)=\sdf(\bx^0)+\sdg(\bx^0)=\sdh(\bx^0)$, we have $\bs^{k}=\sdh(\bx^{k})$ for $k=0,1,\ldots$. Thus, by substituting  $\bs^{k}=\sdh(\bx^{k})$ in~\eqref{eq:up_p1_line1},
		we obtain that 
		the given conditions are equivalent to:}
	\begin{subequations}\label{eq:updates2_proof2_pre}
		\begin{align}
			&\by^k=\by^{k-1}+\Lap\sdh(\bx^{k}),\label{eq:updates2_proof2_pre1}\\
			&B(\sdh(\bx^{k+1})-\sdh(\bx^{k}))+\bx^{k+1}-\br+2c\by^k-c\by^{k-1}=\zero.
			\label{eq:updates2_proof2_pre2}
		\end{align}
	\end{subequations}
	\an{Now, in relation~\eqref{eq:updates2_proof2_pre2}, we write 
		$2c\by^k-c\by^{k-1}=c\by^k +c(\by^k-\by^{k-1})$ and use $\by^k-\by^{k-1}=\Lap\sdh(\bx^{k})$
		(cf.\ \eqref{eq:updates2_proof2_pre1}), }
	to obtain the following equivalent relations: 
	\begin{subequations}\label{eq:updates2_proof2}
		\begin{align}
			&\by^k=\by^{k-1}+\Lap\sdh(\bx^{k}),\label{eq:up_p2_line1}\\
			&B(\sdh(\bx^{k+1})-\sdh(\bx^{k}))+\bx^{k+1}-\br+c\by^k+c\Lap\sdh(\bx^k)=\zero.\label{eq:up_p2_line2}
		\end{align}
	\end{subequations}
	\an{These relations are enough to generate the sequence $\{\bx^k,\by^k\}$.
		By Assumption~\ref{ass:conn_and_prop}, we have that $\Lap=U^\T U$, so by introducing the notation
		$\by^k=U^\T\bq^{k}$, relation \eqref{eq:up_p2_line1} reduces to
		$\bq^{k+1}=\bq^{k}+U\sdh(\bx^{k+1})$. We note that
		a sequence $\{\by^k\}$ generated by $\by^k=U^\T\bq^{k}$ is the same to the sequence $\{\by^k\}$ generated by \eqref{eq:up_p2_line1} with $\by^{-1}=\zero$.
		Using  these relations and reorganizing \eqref{eq:updates2_proof2}, we obtain}
	\begin{subequations}\label{eq:updates2_proof3}
		\begin{align}
			&\bx^{k+1}-\br+cU^\T\bq^k+c\Lap\sdh(\bx^k)+B(\sdh(\bx^{k+1})-\sdh(\bx^{k}))=\zero,\label{eq:up_p3_line1}\\
			&\bq^{k+1}=\bq^{k}+U\sdh(\bx^{k+1}),\label{eq:up_p3_line2}\\
			&\by^k=U^\T\bq^{k}\label{eq:up_p3_line3},
		\end{align}
	\end{subequations}
	which generates the same $\{\bx^k,\by^k\}$ sequence as Algorithm 1 does. 
	\an{Finally, relation \eqref{eq:up_p3_line2} is equivalent to $\bq^{k}=\bq^{k+1}-U\sdh(\bx^{k+1})$,
		which when substituted into \eqref{eq:up_p3_line1} gives relation~\eqref{eq:up_line1}.
		Relations \eqref{eq:up_p3_line2} and \eqref{eq:up_p3_line3} coincide with
		\eqref{eq:up_line2} and \eqref{eq:up_line3}, respectively.}
\end{proof}

If we choose the special case $W=2\tilde{W}-I$ and $\alpha=1/c$ in Corollary 1 
of reference \cite{Shi2015_2}, we will have the recursive relation of P-EXTRA in the following form\footnote{For brevity, we have slightly abused the notation $U$. In P-EXTRA, $U$ is the square root of $I-W$, while in the proposed algorithms here, 
	$U$ is an arbitrary matrix that satisfies $\Lap=U^\T U$. 
	However, in both cases, the null spaces of $U$ are the same, and so are the span spaces of $U^\T$.}:
\begin{equation}\label{eq:PEXTRA}
	\sdh(\bx^{k+1})+cU^\T\bq^{k+1}+c\BW(\bx^{k+1}-\bx^k)=\zero
	\quad\text{ and }\quad\bq^{k+1}=\bq^k+U\bx^{k+1}.
\end{equation}

To fulfill the optimal condition for the resource allocation problem, 
we would need $\sdh(\bx^{k+1})$ to be consensual and the rows of $\bx^{k+1}-\br$ to sum up to $0$. 
In view of the insight from Lemmas~\ref{lem:opc-res} and~\ref{lem:opc-con}, the only thing we need to do is 
to replace $\sdh(\bx^{k+1})$ by $\bx^{k+1}-\br$ and replace $\bx^{k+1}$ by $\sdh(\bx^{k+1})$, as 
\an{discussed in Section~\ref{sec:mirror}}. By doing so, we obtain
\begin{equation}\label{eq:try1}
	\bx^{k+1}-\br+cU^\T\bq^{k+1}+c\BW(\sdh(\bx^{k+1})-\sdh(\bx^k))=\zero
	\quad\text{ and }\quad\bq^{k+1}=\bq^k+U\sdh(\bx^{k+1}),
\end{equation}
which is very similar to the relations \eqref{eq:up_line1} and \eqref{eq:up_line2} in Lemma \ref{lemma:rec_rel_alg1}. \ws{The only difference is in the term $c\BW(\sdh(\bx^{k+1})-\sdh(\bx^k))$ of \eqref{eq:try1} whereas we have ``replaced'' this term by $(B-c\Lap)(\sdh(\bx^{k+1})-\sdh(\bx^k))$ to obtain \eqref{eq:up_line1}. The key function of the term $c\BW(\bx^{k+1}-\bx^k)$ in \eqref{eq:PEXTRA} (corresponding to the term $c\BW(\sdh(\bx^{k+1})-\sdh(\bx^k))$ in~\eqref{eq:try1}) 
	is to stabilize the iterative process and neutralize those terms that are not one-step decentralized implementable in P-EXTRA. Hypothetically, with a ``large enough'' positive (semi)definite matrix $P$, any term $P(\bx^{k+1}-\bx^k)$ in \eqref{eq:PEXTRA} (or $P(\sdh(\bx^{k+1})-\sdh(\bx^k))$ in~\eqref{eq:try1}) will serve the purpose of stabilizing the iterative process. Here, we redesign this term as a more flexible one, $(B-c\Lap)\sdh(\bx^{k+1})$, so that the recursive relations are resolvable and implementable in decentralized manner, while featuring per-agent-independent parameters.}

Our analysis of Algorithm~1 will use the alternative description of the algorithm, as given in Lemma~\ref{lemma:rec_rel_alg1}. 
	To simplify our presentation, let us define the following quantities:
\begin{equation}\label{eq:met_z}
	\met\triangleq\left(
	\begin{array}{cc}
		cI&0\\
		0&B-c\Lap
	\end{array}
	\right),\quad
	\bz^k\triangleq\left(
	\begin{array}{c}
		\bq^k\\
		\sdh(\bx^k)
	\end{array}
	\right),\quad \text{and}\quad
	\bz^*=\bz(\bx^*)\triangleq\left(
	\begin{array}{c}
		\bq^*\\
		\sdh(\bx^*)
	\end{array}
	\right)\ \ \hbox{with }\bx^*\in \bX^*,
\end{equation}
\an{where $c>0$ is the parameter of Algorithm~1.
	Using this particular $c$ in Lemma~\ref{lem:opc-res}, with each solution $\bx^*\in \bX^*$ we can identify $\bq^*$ such that  Lemma~\ref{lem:opc-res} holds, i.e.,
	the optimality conditions in~\eqref{eq:opt_cond} are satisfied. 
	This particular $\bq^*$ and the matrix $\sdh(\bx^*)$ constitute the matrix $\bz^*$.
}

Next, we will show that $\bx^k$ converges to a solution $\bx^*$. 

	\begin{theorem}[Convergence of Mirror-P-EXTRA]\label{theorem:conv}
		Let Assumptions \ref{assum:all}--\ref{ass:sol} hold.  Let the parameters $\beta_i$ and $c>0$ be such that {\color{black} $B-c\Lap\succ\zero$}.
		Then, ${\color{black} \met\succ0,}$ and the sequences $\{\bx^k\}$ and $\{\bq^k\}$ generated by Algorithm~1 
		satisfy the following relations:
		\begin{equation}\label{eq:contraction}
			\|\bz^k-\bz^{k+1}\|_\met^2\leq\|\bz^k-\bz^*\|_\met^2-\|\bz^{k+1}-\bz^*\|_\met^2, \ \ \forall k=0,1,\ldots,
		\end{equation}
		where $\met$, $\bz^k$ and $\bz^*=\bz(\bx^*)$, for $\bx^*\in \bX^*$, are defined by~\eqref{eq:met_z}.
		Furthermore, the sequence $\{\bx^k\}$ converges to a point in the optimal set $\bX^*$.
	\end{theorem}
\begin{proof}
{\color{black} The fact that $\met\succ0$ follows directly 
from the assumptions of the theorem on the choice of the parameters $\beta_i$ and $c>0$.}
		By the convexity of $\bh$, we have that for any arbitrary $\bx^*\in\bX^*$,
		\begin{equation}\label{eq:conv_p1}
			\begin{array}{rcl}
				0&\leq&\langle\sdh(\bx^{k+1})-\sdh(\bx^*),\bx^{k+1}-\bx^*\rangle.
			\end{array}
		\end{equation}	
		By relation~\eqref{eq:up_line1} of
		Lemma~\ref{lemma:rec_rel_alg1} we have
		\[\bx^{k+1}=\br-cU^\T\bq^{k+1}-(B-c\Lap)(\sdh(\bx^{k+1})-\sdh(\bx^k)).\]
		By Lemma~\ref{lem:opc-res}, where $c>0$ is the chosen parameter in the algorithm,
		from \eqref{eq:oc_line1} we have
		\[\bx^*=\br-cU^\T\bq^*.\]
		The preceding two relations imply that
		\[\bx^{k+1}-\bx^*=cU^\T\left( \bq^*-\bq^{k+1}\right)-(B-c\Lap)(\sdh(\bx^{k+1})-\sdh(\bx^k)),\]
		which when substituted into \eqref{eq:conv_p1} yields
		\begin{eqnarray*}
			0&\leq&\langle\sdh(\bx^{k+1})-\sdh(\bx^*),cU^\T(\bq^*-\bq^{k+1})+(B-c\Lap)(\sdh(\bx^k)-\sdh(\bx^{k+1}))\rangle\cr
			&=& 
			\langle U(\sdh(\bx^{k+1})- \sdh(\bx^*)),c(\bq^*-\bq^{k+1})\rangle+
			\langle\sdh(\bx^{k+1})-\sdh(\bx^*),(B-c\Lap)(\sdh(\bx^k)-\sdh(\bx^{k+1}))\rangle.
		\end{eqnarray*}
		Using relation~\eqref{eq:up_line2} of
		Lemma~\ref{lemma:rec_rel_alg1} and $U\sdh(\bx^*)=\zero$ 
		(see Lemma~\ref{lem:opc-res}), it follows that
		\[0\le \langle\bq^{k+1}-\bq^k,c(\bq^*-\bq^{k+1})\rangle+\langle\sdh(\bx^{k+1})-\sdh(\bx^*),(B-c\Lap)(\sdh(\bx^k)-\sdh(\bx^{k+1}))\rangle.
		\]
		Recalling the definitions of $\met$, $\bz^k$, and $\bz^*$ in \eqref{eq:met_z}, by applying the basic equality $2\langle\bz^{k+1}-\bz^k,\met(\bz^*-\bz^{k+1})\rangle=\|\bz^k-\bz^*\|_\met^2-\|\bz^{k+1}-\bz^*\|_\met^2-\|\bz^k-\bz^{k+1}\|_\met^2$,  from the preceding inequality we obtain
		\begin{equation}\label{eq:conv_p3}
			\begin{array}{rcl}
				\|\bz^k-\bz^{k+1}\|_\met^2 + \|\bz^{k+1}-\bz^*\|_\met^2\leq\|\bz^k-\bz^*\|_\met^2.
			\end{array}
		\end{equation}
	
	By summing relations in~\eqref{eq:conv_p3} over $k$ from $k=0$ to $k=t$ for some $t>0$ and, 
		then, letting $t\to\infty$, we find that 
		$\sum_{k=0}^\infty \|\bz^k-\bz^{k+1}\|_\met^2 <\infty$, which implies that
		$\lim_{k\to\infty}\|\bz^k-\bz^{k+1}\|_\met^2=0$. Consequently, by the definition of $\bz^k$ and $ \met$, 
		we have that $\lim_{k\to\infty}\|\bq^k-\bq^{k+1}\|_\Fro^2=0$ and 
		$\lim_{k\to\infty}\|\sdh(\bx^k)-\sdh(\bx^{k+1})\|_{B-c\Lap}^2=0$.
Using relations~ \eqref{eq:up_line1} and \eqref{eq:up_line2} from Lemma \ref{lemma:rec_rel_alg1}, 
		we obtain $\|\bq^k-\bq^{k+1}\|_\Fro^2=\|U\sdh(\bx^{k+1})\|_\Fro^2$ and 
		$\|\bx^{k+1}-\br+cU^\T\bq^{k+1}\|_\Fro^2\leq\laml{B-c\Lap}\|\sdh(\bx^k)-\sdh(\bx^{k+1})\|_{B-c\Lap}^2$, thus yielding
		\begin{subequations}\label{eq:conv_p4}
			\begin{align}
				&\lim_{k\rightarrow\infty}U\sdh(\bx^k)=\zero,\label{eq:conv_p4_line1}\\
				&\lim_{k\rightarrow\infty}\{\bx^k-\br+cU^\T\bq^k\}=\zero.\label{eq:conv_p4_line2}
			\end{align}
		\end{subequations}
		From \eqref{eq:conv_p3} we have that 
		\begin{equation}\label{eq:conv_p5} \|\bz^{k+1}-\bz^*\|_\met^2\leq\|\bz^k-\bz^*\|_\met^2,\quad\forall k\ge0,
		\end{equation}
		implying that
	{\color{black} the sequences $\{\bq^k\}$ and $\{\sdh(\bx^k)\}$ are bounded (recall the definition of $\bz^k$ and $\met$ in~\eqref{eq:met_z} and the fact that $\met\succ0$).} 
		In view of~\eqref{eq:conv_p4_line2} it follows that $\{\bx^k\}$ is also bounded and, moreover, that 
		the sequences $\{\bq^k\}$ and $\{\bx^k\}$ have the same convergent subsequences.
		In particular, let $\{\bx^{k_\ell}\}$ and $\{\bq^{k_\ell}\}$ be convergent subsequences, and 
		{\color{black}  
		let $\tilde\bx$ and $\tilde\bq$ be their limit points, respectively. 
		Since $\{\sdh(\bx^k)\}$ is bounded, without loss of generality we may assume that 
		$\{\sdh(\bx^{k_\ell})\}$ converges to some matrix $\bd$
		(for otherwise, we would select a convergent subsequence of $\{\sdh(\bx^{k_\ell})\})$.			
		Then, by letting $\ell\to\infty$, from relations~\eqref{eq:conv_p4_line1} and \eqref{eq:conv_p4_line2}, respectively, we have
		\begin{equation}\label{eq:fin}
			U\bd=\zero,
	\end{equation}
		\begin{equation}\label{eq:almost} 
			\tilde \bx - \br+cU^\T\tilde \bq=\zero.\end{equation}
			Recalling that  $\{\sdh(\bx)\}$ is the matrix with the $i$th row consisting of a subgradient 
		$\widetilde \nabla h_i(x_i)$ of $h_i$ at $x_i$, 
		we have for all $i$,
		\[h_i(\tilde x_i^{k_\ell}) + \langle \tilde \nabla h_i(x_i^{k_\ell}),x_i^{k_\ell}-x_i\rangle
		\le h_i(x_i)\qquad\hbox{for all }x_i.\]
		By letting $\ell\to\infty$, and using $\bx^{k_\ell}\to \tilde \bx$ and $\sdh(\bx^{k_\ell})\to \bd$, we obtain
		for all $i$,
		\[h_i(\tilde \bx_i) + \langle d_i,\tilde \bx_i - \bx_i\rangle\le h_i(\bx_i)\qquad\hbox{for all }\bx_i,\]
		showing that, for every $i$, the vector $d_i$ is a subgradient of $h_i$ at $\tilde x_i$, i.e., $d=\sdh(\tilde\bx)$.
		Thus, relations \eqref{eq:fin} and~\eqref{eq:almost} show that $(\tilde \bx,\tilde \bq)$ satisfies the optimality conditions of Lemma~\ref{lem:opc-res}.}
		Therefore, by Lemma~\ref{lem:opc-res}, it follows that $\tilde \bx$ is an optimal solution to the resource allocation problem, i.e., $\tilde \bx\in \bX^*$.
	
	\an{We now define 
		\[\tilde \bz^*\triangleq\left(
		\begin{array}{c}
		\tilde \bq\\
		\sdh(\tilde\bx)
		\end{array}
		\right),\]
		and use $\bz^*=\tilde \bz^*$ in relation~\eqref{eq:conv_p5}, together with the fact that 
		$\bz^{k_\ell}\to
		\tilde \bz^*$, to conclude that the entire sequence $\{\bz^{k}\}$ must converge to $\tilde \bz^*$. By the 
		definition of $\bz^k$, it follows that
		$\bq^k\to\tilde \bq$, which when used in equation~\eqref{eq:conv_p4_line2} implies that
		\[\lim_{k\rightarrow\infty}\bx^k-\br+cU^\T\tilde\bq=\zero.\]
		The preceding relation and~\eqref{eq:almost} yield $\bx^k\to\tilde \bx$.
	}
\end{proof}

To establish the rate of convergence, we use a convergence property of nonnegative monotonic scalar sequence,
which has appeared in recent work~\cite{Deng2016,Shi2015}. The result is stated in the following:

\begin{proposition}\label{prop:o_1_k}
	If a sequence $\{a_k\}\subset\R$ is such that
	$a_k\geq0,$ $\sum_{t=1}^{\infty}a_t<\infty$, and $a_{k+1}\leq a_k$ for all $k$,
	then $a_k=o\left(\frac{1}{k}\right).$
\end{proposition}

\begin{lemma}[Monotonic successive difference of Mirror-P-EXTRA]\label{lemma:mono}
	Under Assumptions~\ref{assum:all} and~\ref{ass:conn_and_prop}, we have
	\begin{equation}\label{eq:monotonicity}
		\|\bz^{k+1}-\bz^{k+2}\|_\met^2\leq\|\bz^k-\bz^{k+1}\|_\met^2,\ \forall k=0,1,\ldots.
	\end{equation}
\end{lemma}

\begin{proof}
	To simplify the notation, let us define
	$\D\bx^{k+1}\triangleq\bx^k-\bx^{k+1}$,
	$\D\bq^{k+1}\triangleq\bq^k-\bq^{k+1}$,
	$\D\bz^{k+1}\triangleq\bz^k-\bz^{k+1}$, and
	$\D\sdh(\bx^{k+1})\triangleq\sdh(\bx^k)-\sdh(\bx^{k+1})$. By the convexity of $\bh$, we have
	\begin{equation}\label{eq:mono_proof_1}
		\begin{array}{rcl}
			\langle\D\bx^{k+1},\D\sdh(\bx^{k+1})\rangle\geq0.
		\end{array}
	\end{equation}
	\an{Next, we use Lemma~\ref{lemma:rec_rel_alg1}, where 
		by taking the difference of \eqref{eq:up_line1} at the $k$-th and
		the $(k+1)$-at iteration, we obtain}
	\begin{equation}\label{eq:mono_proof_2}
		\begin{array}{c}
			\D\bx^{k+1}+cU^\T\D\bq^{k+1}+(B-c\Lap)(\D\sdh(\bx^{k+1})-\D\sdh(\bx^k))=\zero.
		\end{array}
	\end{equation}
	By combining~\eqref{eq:mono_proof_1} and \eqref{eq:mono_proof_2}, it
	follows that
	\begin{equation}\label{eq:mono_proof_3}
		\begin{array}{rl}
			\langle-cU^\T\D\bq^{k+1}-(B-c\Lap)(\D\sdh(\bx^{k+1})-\D\sdh(\bx^k)),\D\sdh(\bx^{k+1})\rangle\geq0.
		\end{array}
	\end{equation}
	Using the relation \eqref{eq:up_line2} for $\bq^k$, we see that 
	$\D\bq^{k+1}=-U\sdh(\bx^{k+1})$. Thus, we have
	\begin{equation*}
		\D\bq^{k}-\D\bq^{k+1}=-U\D\sdh(\bx^{k+1}),
	\end{equation*}
	which when substituted into \eqref{eq:mono_proof_3} and re-arranging some terms
	yields 
	\begin{equation*}
		\begin{array}{c}
			\langle c\D\bq^{k+1},\D\bq^{k}-\D\bq^{k+1}\rangle+\langle (B-c\Lap)\D\sdh(\bx^{k+1}),\D\sdh(\bx^k)-\D\sdh(\bx^{k+1})\rangle\geq0,
		\end{array}
	\end{equation*}
	or equivalently
	\begin{equation}\label{eq:mono_proof_6}
		\begin{array}{rcl}
			\langle \met\D\bz^{k+1},\D\bz^{k}-\D\bz^{k+1}\rangle\geq 0.
		\end{array}
	\end{equation}
	By applying the basic equality $2\langle \met\D\bz^{k+1},\D\bz^{k}-\D\bz^{k+1}\rangle=\|\D\bz^k\|_\met^2-\|\D\bz^{k+1}\|_\met^2-\|\D\bz^k-\D\bz^{k+1}\|_\met^2$
	to \eqref{eq:mono_proof_6}, we finally have
	\begin{equation*}
		\begin{array}{rcl}
			\|\D\bz^k\|_\met^2-\|\D\bz^{k+1}\|_\met^2
			&\geq&\|\D\bz^k-\D\bz^{k+1}\|_\met^2\\
			&\geq&0,
		\end{array}
	\end{equation*}
	which implies \eqref{eq:monotonicity} and completes the proof.
\end{proof}

Based on Lemma~\ref{lemma:mono}, we have the following rate result for the first-order optimality residual. 

\begin{theorem}[Sublinear rate of Mirror-P-EXTRA]\label{theorem:o_1_k}
	Under the assumptions of Theorem~\ref{theorem:conv}, the first-order optimality residual decays to $0$ at an $o\left(\frac{1}{k}\right)$ rate, i.e.,
	\[\|U\sdh(\bx^{k+1})\|_\Fro^2=o\left(\frac{1}{k}\right),\qquad
	\|\bx^{k+1}-\br+cU^\T\bq^{k+1}\|_{\Fro}^2=o\left(\frac{1}{k}\right).\]
\end{theorem}
\begin{proof}
	As an immediate consequence of 
	Theorem~\ref{theorem:conv} and Proposition~\ref{prop:o_1_k}, we can see that 
	Lemma~\ref{lemma:mono} holds for the sequence $a_k=\|\bz^k-\bz^{k+1}\|_\met^2$, which implies 
	the stated result.
\end{proof}

{\color{black} Using Theorem~\ref{theorem:o_1_k}, we can obtain the vanishing speed
for the violation of the global constraint $\one^\T(\bx-\br)=0$ by noticing that
$\|\one^\T(\bx^k-\br)\|_\Fro^2=\|\bx^k-\br+cU^\T\bq^{k}\|_{\one\one^\T}^2=o(1/k)$. 
In reference~\cite{aybat2016distributed2}, an ergodic rate $\|\one^\T(\bar{\bx}^k-\br)\|=O(1/k)$ 
is derived for the running average  $\bar{\bx}^k=\sum_{t=0}^k\bx^t/(1+k)$. 
However, we would like to point out that our bound should not be considered worse than that has appeared in \cite{aybat2016distributed2} since our rate is non-ergodic.
}

\an{Under additional assumptions on the objective function, 
	we can provide stronger convergence rate results. In particular, we will
	show a geometric convergence rate for the case when $\Omg$ is the full space, i.e, 
	the problem~\eqref{eq:basic} has no per-agent set constraints, in which case $\bh=\f$. 
	The geometric rate is obtained under strong convexity and the Lipschitz gradient property of the objective function $\f$. These properties are defined as follows.
	A function $f:\R^p\rightarrow\R$ is said to be $\mu$-strongly convex if $\langle\widetilde{\nabla}f(x)-\widetilde{\nabla}f(y),x-y\rangle\geq\mu\|x-y\|_2^2$ for some $\mu>0$ and for all $x,y\in\R^p$. 
	A gradient map $\nabla f:\R^p\rightarrow\R^p$ is said to be $L$-Lipschitz if 
	$\|\nabla f(x)-\nabla f(y)\|_2\leq L\|x-y\|_2$ for some $L>0$ and for all $x,y\in\R^p$. 
	We formally impose these properties in the following assumptions.}

\begin{assumption}\label{ass:asss_for_geo1}
	The gradient \an{map $\nabla f_i$} of each function $f_i:\R^p\to\R$ is $L_i$-Lipschitz.
\end{assumption}
\begin{assumption}\label{ass:asss_for_geo2}
	Each function $f_i$ is $\mu_i$-strongly convex.
\end{assumption}
Let us \an{introduce} $\mu=\min_i \mu_i$ and $L=\max_i L_i$.

\begin{theorem}[Linear rate of Mirror-P-EXTRA]\label{theorem:geo} 
	\an{
		Let $\Omega_i=\R^p$ for all $i$, and 
		let Assumption~\ref{ass:conn_and_prop},~\ref{ass:asss_for_geo1} and~\ref{ass:asss_for_geo2} hold.
		Then, for Algorithm 1 with the matrix parameters such that $B\succcurlyeq c\Lap$, 
		the sequence $\{\bz^k\}$ converges to $\bz^*$ with Q-linear rate in the (pseudo)metric induced by the matrix $\met$, where $\bz^k$, $\bz^*$ and $\met$ are defined in~\eqref{eq:met_z}. 
		In particular, we have
		$\|\bz^{k+1}-\bz^*\|_\met^2\leq\frac{1}{1+\delta}\|\bz^{k}-\bz^*\|_\met^2$ with some $\delta>0$. 
		Consequently, the sequence $\{\bx^k\}$ converges to the optimal solution $\bx^*$ of problem~\eqref{eq:basic}
		at an R-linear rate $O\left(\left(\frac{1}{1+\delta}\right)^k\right)$.
	}
\end{theorem}
\begin{proof}
	Note that we have $\f=\bh$, so that
	by the strong convexity and the gradient Lipschitz property of $\bh$, it follows that \cred{(see Theorem 2.1.11 of \cite{nesterov2013introductory})}
	\begin{equation}\label{eq:geo_p1}
		\begin{array}{rcl}
			\frac{\mu L}{\mu+L}\|\bx^{k+1}-\bx^*\|_\Fro^2+\frac{1}{\mu+L}\|\dbh(\bx^{k+1})-\dbh(\bx^*)\|_\Fro^2&\leq&\langle\dbh(\bx^{k+1})-\dbh(\bx^*),\bx^{k+1}-\bx^*\rangle.
		\end{array}
	\end{equation}	
	Using arguments similar to those from \eqref{eq:conv_p1} to \eqref{eq:conv_p3}, where \eqref{eq:conv_p1} is replaced by \eqref{eq:geo_p1}, we obtain
	\begin{equation}\label{eq:geo_p2}
		\begin{array}{rcl}
			&    &\frac{\mu L}{\mu+L}\|\bx^{k+1}-\bx^*\|_\Fro^2+\frac{1}{\mu+L}\|\dbh(\bx^{k+1})-\dbh(\bx^*)\|_\Fro^2+\|\bz^k-\bz^{k+1}\|_\met^2\\
			&\leq&\|\bz^k-\bz^*\|_\met^2-\|\bz^{k+1}-\bz^*\|_\met^2.
		\end{array}
	\end{equation}
	To prove the linear convergence, 
	\an{it suffices to show that} $\delta\|\bz^{k+1}-\bz^*\|_\met^2\leq \|\bz^k-\bz^*\|_\met^2-\|\bz^{k+1}-\bz^*\|_\met^2$ for some $\delta>0$. Considering \eqref{eq:geo_p2}, this will hold as long as for some $\delta>0$ and all $k\ge0$ we have
	\begin{equation}\label{eq:geo_p3}
		\begin{array}{rcl}
			\delta\|\bz^{k+1}-\bz^*\|_\met^2\leq\frac{\mu L}{\mu+L}\|\bx^{k+1}-\bx^*\|_\Fro^2+\frac{1}{\mu+L}\|\dbh(\bx^{k+1})-\dbh(\bx^*)\|_\Fro^2+\|\bz^k-\bz^{k+1}\|_\met^2.
		\end{array}
	\end{equation}
	\an{By the definition of $\bz^k$, we have 
		\begin{equation}\label{eq:step1}
			\|\bz^k-\bz^{k+1}\|_\met^2=c\|\bq^{k+1}-\bq^k\|_\Fro^2 +\|\dbh(\bx^k)-\dbh(\bx^{k+1})\|_{B-c\Lap}^2,
		\end{equation}
		while by Lemma~\ref{lemma:rec_rel_alg1} we have
		\begin{equation}\label{eq:step2}
			\|\bq^{k+1}-\bq^k\|_\Fro^2= \|U\nabla \bh(\bx^{k+1})\|_\Fro^2=
			\|U\nabla \bh(\bx^{k+1})-U\nabla \bh(\bx^*)\|_\Fro^2,
		\end{equation}
		where in the last equality we use the fact $U\nabla \bh(\bx^*)=\zero$ (see~\eqref{eq:oc_line2}).
		Substituting~\eqref{eq:step1} and~\eqref{eq:step2} in~\eqref{eq:geo_p3}, we conclude that for the linear convergence, it is sufficient to show the following relation:}
	\begin{equation}\label{eq:geo_p4}
		\begin{array}{rcl}
			&& \delta c\|\bq^{k+1}-\bq^*\|_\Fro^2+\delta\|\dbh(\bx^{k+1})-\dbh(\bx^*)\|_{B-c\Lap}^2
			\leq \frac{\mu L}{\mu+L}\|\bx^{k+1}-\bx^*\|_\Fro^2 +
			\frac{1}{\mu+L}\|\dbh(\bx^{k+1})-\dbh(\bx^*)\|_\Fro^2\cr
			&&+c\|U(\dbh(\bx^{k+1})-\dbh(\bx^*))\|_\Fro^2+\|\dbh(\bx^k)-\dbh(\bx^{k+1})\|_{B-c\Lap}^2.\qquad\qquad
		\end{array}
	\end{equation}
	
	\an{
		From \eqref{eq:oc_line1} and \eqref{eq:up_line1} we have
		\begin{equation}\label{eq:oc_up_line1}	
			\bx^{k+1}-\bx^*+cU^\T(\bq^{k+1}-\bq^*)+(B-c\Lap)(\dbh(\bx^{k+1})-\dbh(\bx^k))=\zero,
		\end{equation}
		implying that 
		\begin{eqnarray}\label{eq:geo_p4-1}
			\|cU^\T(\bq^{k+1}-\bq^*)\|^2_\Fro 
			& = & \|\bx^{k+1}-\bx^*+(B-c\Lap)(\dbh(\bx^{k+1})-\dbh(\bx^k)\|^2_\Fro\cr
			&\le & (1+\gamma)\|\bx^{k+1}-\bx^*\|_\Fro^2
			+\left(1+\frac{1}{\gamma}\right)\|\dbh(\bx^{k+1})-\dbh(\bx^k)\|_{(B-c\Lap)^2}^2,
		\end{eqnarray}
		where $\gamma>0$ is arbitrary.
	}
	\ws{Since the matrix $U\in\R^{(n-1)\times n}$ is full row rank, we have
		\[ c^2\lams{UU^\T}\|\bq^{k+1}-\bq^*\|^2_\Fro \le \|cU^\T(\bq^{k+1}-\bq^*)\|^2_\Fro.\]
		This together with  $\lams{UU^\T}=\lams{\Lap}$, 
		when substituted in~\eqref{eq:geo_p4-1} 
		yields
		\begin{equation}\label{eq:geo_p5}
			\begin{array}{rcl}
				c^2\lams{\Lap}\|\bq^{k+1}-\bq^*\|_\Fro^2\leq(1+\gamma)\|\bx^{k+1}-\bx^*\|_\Fro^2+\left(1+\frac{1}{\gamma}\right)\|\dbh(\bx^{k+1})-\dbh(\bx^k)\|_{(B-c\Lap)^2}^2.
			\end{array}
	\end{equation}}
	\an{Combining \eqref{eq:geo_p4} and \eqref{eq:geo_p5},
		we find that, in order to establish the rate result for $\|\bz^{k+1}-\bz^*\|_\met^2$, 
		it suffices to show that}
	\begin{equation}\label{eq:geo_p6}
		\begin{array}{rl}
			&\frac{\delta}{c\lams{\Lap}}\left(1+\frac{1}{\gamma}\right)\|\dbh(\bx^{k+1})-\dbh(\bx^k)\|_{(B-c\Lap)^2}^2+\|\dbh(\bx^{k+1})-\dbh(\bx^*)\|_{\delta(B-c\Lap)-\frac{1}{\mu+L}I-c\Lap}^2\\
			\leq&
			\|\dbh(\bx^k)-\dbh(\bx^{k+1})\|_{B-c\Lap}^2+\left(\frac{\mu L}{\mu+L}-\frac{\delta}{c\lams{\Lap}}(1+\gamma)\right) \|\bx^{k+1}-\bx^*\|_\Fro^2.
		\end{array}
	\end{equation}
	
	Now, we utilize the basic properties of function $\bh$ to verify the validness of \eqref{eq:geo_p6}. Considering the Lipschitz continuity of $\dbh$,	we only need to determine a positive $\delta$ such that
\begin{subequations}\label{eq:geo_p7_ori}
	\begin{numcases}{}
			\frac{\delta}{c\lams{\Lap}}\left(1+\frac{1}{\gamma}\right)(B-c\Lap)\preccurlyeq I,\label{eq:geo_p7_ori_ln1}\\
			\hbox{}\cr
			\left(\delta(B-c\Lap)-\frac{1}{\mu+L}I-c\Lap\right)L^2\preccurlyeq\left(\frac{\mu L}{\mu+L}-\frac{\delta}{c\lams{\Lap}}(1+\gamma)\right)I,\label{eq:geo_p7_ori_ln2}\\
			\hbox{}\cr
			0\preccurlyeq\left(\frac{\mu L}{\mu+L}-\frac{\delta}{c\lams{\Lap}}(1+\gamma)\right)I.\label{eq:geo_p7_ori_ln3}
			\end{numcases}
\end{subequations}
	or 
\begin{subequations}\label{eq:geo_p8_ori}
	\begin{numcases}{}
		\frac{\delta}{c\lams{\Lap}}\left(1+\frac{1}{\gamma}\right)(B-c\Lap)\preccurlyeq I,\label{eq:geo_p8_ori_ln1}\\
		\hbox{}\cr
		-\mu^2\left(\delta(B-c\Lap)-\frac{1}{\mu+L}I-c\Lap\right)\succcurlyeq-\left(\frac{\mu L}{\mu+L}-\frac{\delta}{c\lams{\Lap}}(1+\gamma)\right)I,\label{eq:geo_p8_ori_ln2}\\
		\hbox{}\cr
		-\left(\delta(B-c\Lap)-\frac{1}{\mu+L}I-c\Lap\right)\succcurlyeq 0. \label{eq:geo_p8_ori_ln3}
	\end{numcases}{}.
\end{subequations}

Let us use $\lhs{i}$ and $\rhs{i}$ (see Section \ref{sec:notation}) to refer to key terms appeared in \eqref{eq:geo_p6} and elaborate how \eqref{eq:geo_p7_ori} and \eqref{eq:geo_p8_ori} are obtained. The (matrix) inequality \eqref{eq:geo_p7_ori_ln1} and \eqref{eq:geo_p8_ori_ln1} are identical and due to requiring $\lhs{1}\leq\rhs{1}$; the matrix inequalities \eqref{eq:geo_p7_ori_ln2} and \eqref{eq:geo_p7_ori_ln3} are due to using the Lipschitz continuity of the gradient to generate a lower bound for $\rhs{2}$ and requiring $\lhs{2}$ bing bounded by it; the matrix inequalities \eqref{eq:geo_p8_ori_ln2} and \eqref{eq:geo_p8_ori_ln3} are due to using the strong convexity of $\bh$ to generate an upper bound for $-\rhs{2}$, i.e.,
\begin{equation}
\begin{array}{rcl}
\mu\|\bx^{k+1}-\bx^*\|_\Fro^2&\leq&\langle\bx^{k+1}-\bx^*,\dbh(\bx^{k+1})-\dbh(\bx^*)\rangle\\
&\leq&\|\bx^{k+1}-\bx^*\|\cdot\|\dbh(\bx^{k+1})-\dbh(\bx^*)\|,
\end{array}
\end{equation}
and requiring this upper bound being further bounded by $-\lhs{2}$.

\an{By re-arranging terms and shrinking the feasible regions of \eqref{eq:geo_p7_ori} and \eqref{eq:geo_p8_ori}, we re-write these conditions in more compact forms as follows:
		\begin{subequations}\label{eq:geo_p7}
			\begin{numcases}{}
				\frac{\delta\laml{B-c\Lap}}{c\lams{\Lap}}\left(1+\frac{1}{\gamma}\right)\leq1,\label{eq:geo_p7_l1}\\
				\delta L^2 (B-c\Lap)+\frac{\delta}{c\lams{\Lap}}(1+\gamma)I
				\preccurlyeq L I+ cL^2 \Lap,\label{eq:geo_p7_l2}\\
				\frac{\delta}{c\lams{\Lap}}(1+\gamma)\le \frac{\mu L}{\mu+L},\label{eq:geo_p7_l3}
			\end{numcases}
		\end{subequations}
		\begin{subequations}\label{eq:geo_p8}
			\begin{numcases}{}
			\frac{\delta\laml{B-c\Lap}}{c\lams{\Lap}}\left(1+\frac{1}{\gamma}\right)\leq1,\label{eq:geo_p8_l1}\\
			\delta\mu^2(B-c\Lap) +\frac{\delta}{c\lams{\Lap}}(1+\gamma)I\preccurlyeq
			\mu I + c\mu^2\Lap,\label{eq:geo_p8_l2}\\
			\delta(B-c\Lap) \preccurlyeq\frac{1}{\mu+L}I+c\Lap.\label{eq:geo_p8_l3}
			\end{numcases}
		\end{subequations}
		For any $\gamma>0$, a positive small enough $\delta$ exists that satisfies either $\delta\leq\delta_1=\sup\{\delta>0\mid\delta\text{ satisfies \eqref{eq:geo_p7} }\}$ or $\delta\leq\delta_2=\sup\{\delta>0\mid \delta\text{ satisfies \eqref{eq:geo_p8} }\}$. Thus, we can find a positive $\delta\leq\max\{\delta_1,\delta_2\}$ such that $\|\bz^{k+1}-\bz^*\|_\met^2\leq\frac{1}{1+\delta}\|\bz^{k}-\bz^*\|_\met^2$.}
	
	\ws{By the definition of $\bz^k$, $\bz^*$, and $\met$, the above Q-linear rate of $\{\|\bz^{k}-\bz^*\|_\met^2\}$ immediately translates into the R-linear rates:
		$\|\bq^k-\bq^*\|_\Fro^2=O\left(\left(\frac{1}{1+\delta}\right)^k\right)$ and $\|\dbh(\bx^k)-\dbh(\bx^*)\|_{B-c\Lap}^2=O\left(\left(\frac{1}{1+\delta}\right)^k\right)$.
		From \eqref{eq:oc_up_line1}, we have
		\begin{equation}\label{eq:geo_p9}	
			\begin{array}{rcl}
				\|\bx^{k+1}-\bx^*\|_\Fro^2
				&=&O(\|\bq^{k+1}-\bq^*\|_\Fro^2)+O(\|(B-c\Lap)(\dbh(\bx^{k+1})-\dbh(\bx^k))\|_\Fro^2)\\
				&=&O(\|\bq^{k}-\bq^*\|_\Fro^2)+O(\|\dbh(\bx^{k})-\dbh(\bx^*)\|_{B-c\Lap}^2)\\
				&=&O\left(\left(\frac{1}{1+\delta}\right)^k\right).		
			\end{array}
	\end{equation}}
\end{proof}

In the above proof of Theorem \ref{theorem:geo}, explicit bounds on $\delta$ can be easily derived when we further choose $B=\beta I$. In this case, a sufficient condition on $\delta$ would be either
\begin{equation}\label{eq:final_delta}
	\delta\leq\delta_1\triangleq\min\left\{\frac{c\lams{\Lap}}{\beta(1+\frac{1}{\gamma})},\frac{cL\lams{\Lap}}{\beta cL^2\lams{\Lap}+1+\gamma},\frac{\mu Lc\lams{\Lap}}{(\mu+L)(1+\gamma)}\right\},
\end{equation}
or
\begin{equation}\label{eq:final_delta2}
\delta\leq\delta_2\triangleq\min\left\{\frac{c\lams{\Lap}}{\beta(1+\frac{1}{\gamma})},\frac{c\mu\lams{\Lap}}{\beta c\mu^2\lams{\Lap}+1+\gamma},\frac{1}{(\mu+L)\beta}\right\},
\end{equation}
where $\gamma>0$ is a parameter that we can choose. 
Consequently, we have the following two corollaries.

\begin{corollary}[The scalability of Mirror-P-EXTRA]\label{corollary:scalability} 
	Under the conditions of Theorem~\ref{theorem:geo}, and further choosing $B=\beta I$, with appropriate parameter settings in Mirror-P-EXTRA, to reach $\varepsilon$-accuracy, 
	the number of iterations needed is of the order $O\left((\kappa_{\Lap}+\sqrt{\kappa_{\Lap}\kappa_\f})\ln\left(\frac{1}{\varepsilon}\right)\right)$, where $\kappa_\f=\frac{L}{\mu}$ is the condition number of 
	the objective function $\f$ and $\kappa_\Lap=\frac{1}{\lams{\Lap}}$. 
\end{corollary}

\begin{proof}  
	Since it appears in both \eqref{eq:geo_p7} that the larger $\beta$ is, the worse the convergence rate is, let us choose $\beta=c$.
	
	From \eqref{eq:final_delta}, we know that to reach $\varepsilon$-accuracy, the number of iterations needed is (omitted the factor $\ln\left(\frac{1}{\varepsilon}\right)$ for briefness)
	\begin{equation}\label{eq:comp_p1}
		\begin{array}{rcl}
			&&O\left(\frac{\beta(1+\frac{1}{\gamma})}{c\lams{\Lap}}\right)+O\left(\frac{\beta cL^2\lams{\Lap}+1+\gamma}{cL\lams{\Lap}}\right)+O\left(\frac{(\mu+L)(1+\gamma)}{\mu Lc\lams{\Lap}}\right)\\
			&=&O\left(\frac{1+\frac{1}{\gamma}}{\lams{\Lap}}\right)+O\left(cL\right)+O\left(\frac{1+\gamma}{\mu c\lams{\Lap}}\right).
		\end{array}
	\end{equation}
	By choosing $\gamma=\mu c$ and setting an appropriate parameter $c=\Theta\left(\frac{1}{\sqrt{\mu L\lams{\Lap}}}\right)$, \eqref{eq:comp_p1} can be further simplified as $O\left(\frac{1}{\lams{\Lap}}\right)+O\left(\sqrt{\frac{L}{\mu}}\sqrt{\frac{1}{\lams{\Lap}}}\right)$. Thus, we conclude that the final complexity is 
	$O\left((\kappa_{\Lap}+\sqrt{\kappa_{\Lap}\kappa_{\f}})\ln\left(\frac{1}{\varepsilon}\right)\right)$, where 
	$\kappa_{\f}=\frac{L}{\mu}$ is the condition number of the objective function $\f$  
	and $\kappa_\Lap=\frac{1}{\lams{\Lap}}$ is understood as the condition number of the graph. 
\end{proof}

\begin{corollary}[Parameters selection and agent number dependency]\label{cor:two}
	Under the same conditions as those in Corollary \ref{corollary:scalability}, \ws{by letting
		$\Lap=0.5(I-W)$ where $W$ is a lazy Metropolis matrix, that is,
		\[
		\wdo_{ij}=\left\{
		\begin{array}{ll}
		1/\left(2\max\{|\Ni|,|\Nj|\}\right),     &\text{ if } ({j,i})\in\E,\\
		0,                              &\text{ if } ({j,i})\notin\E \text{ and } j\neq i,\\
		1-\sum_{l\in\Ni}\wdo_{il},&\text{ if } j=i,
		\end{array}
		\right.
		\]
		and by setting the parameters in Mirror-P-EXTRA as $c=\Theta\left(n\mu^{-0.5}L^{-0.5}\right)$ and $\beta=c$, to reach $\varepsilon$-accuracy, the number of iterations needed is of the order of $O\left((n^2+n\kappa_\f^{0.5})\ln\left(\varepsilon^{-1}\right)\right)$.} 
\end{corollary}

\ws{We omit the proof of Corollary~\ref{cor:two} since it is just the consequence of 
	the fact $\kappa_\Lap=O(n^2)$ under the choice\footnote{\ws{When $W$ is the lazy Metropolis matrix, $\max\{|\laml{W}|,|\lams{W}|\}\leq\sigl{W-\one\one^\T/n}\leq 1-\frac{2}{(71n^2)}$ (see Lemma 2.2 of reference \cite{Olshevsky2014}). Thus $\lams{0.5(I-W)}\geq 1/(71n^2)$ and therefore $\kappa_\Lap=O(n^2)$.}} $\Lap=0.5(I-W)$ and following the argument in the proof of Corollary \ref{corollary:scalability}. For more bounds on $\kappa_\Lap$ which will yield to different scalabilities for other classes of graphs, the readers are refereed to Remark 2 of reference \cite{Nedich2016}.}

\cred{In the scenario when $\bh=\f+\g$ is strongly convex but not necessarily smooth (general convex per-agent constraints $\Omega_i$'s are allowed), one may be able to obtain $O(1/k^2)$ convergence rate on $\|\bx^k-\bx^*\|_\Fro^2$. But this is beyond the coverage of our current paper and will be one of our future studies.}
		
Although Algorithm 1 has fast convergence rate and scales favorably with the condition number of the graph 
and the objective function, 
the computational cost per-iteration is high. In the next section, we propose a gradient-based scheme which 
is specialized for the unconstrained case ($\Omega_i=\R^p$) and has a low per-iteration cost. 

\subsection{The algorithm specialized for the unconstrained case: Mirror-EXTRA}
In this section we concern the resource allocation without local constraints, i.e.,
\begin{equation}\label{eq:F2}
	\begin{array}{cl}
		\min\limits_{\bx} & \f(\bx)=\sum\limits_{i=1}^n
		f_i(x_i),\\
		\st & \one^\T(\bx-\br)=\zero.
	\end{array}
\end{equation}

Our proposed Algorithm 2 applies to \eqref{eq:F2} when the objective function has Lipschitz gradients,
and it is given as follows.
\newpage
\begin{center}
	{\textbf{Algorithm 2: Mirror-EXTRA}}
	
	\smallskip
	\begin{tabular}{l}
		\hline
		\emph{  } A parameter $c$ is selected such that $c\in\left(0,\frac{1}{2L\laml{\Lap}}\right)$;\\
		\emph{  } Each agent $i$ initializes with $x_i^0\in\R^p$ and $y_i^{-1}=0$;\\
		\emph{  } Each agent $i$ \textbf{for} $k=0,1,\ldots$ \textbf{do} \\
		\qquad$y_i^{k}=y_i^{k-1}+\sum_{j\in\Ni\cup\{i\}}\text{\L}_{ij}\dfj(x_j^{k})$;\\
		\qquad$x_i^{k+1}=r_i-2cy_i^k+cy_i^{k-1}$;\\
		\emph{  } \textbf{end}\\
		\hline
	\end{tabular}
\end{center}
\smallskip

\an{Compared to Algorithm 1, Algorithm 2 has a lower per-iteration cost in the update of $x_i^k$ which requires
	a gradient evaluation instead of ``prox"-type (minimization) update.  We next provide an alternative description
	Algorithm 2 that we use in the analysis of the method. } 

\begin{lemma}\label{lemma:rec_rel_alg2}
	\an{Let $\Omega_i=\R^p$ for all $i$ and let Assumptions~\ref{assum:all}, \ref{ass:conn_and_prop}, and~\ref{ass:asss_for_geo1} be satisfied.}
	Then, the sequence $\{\bx^k,\by^k\}$ generated by Algorithm 2 satisfies the following relations:
	\an{for any optimal solution $\bx^*\in\bX^*$ and for all} $k=0,1,\ldots$,
	\begin{subequations}\label{eq:updates3}
		\begin{align}
			&\bx^{k+1}-\br+cU^\T\bq^{k+1}-c\Lap(\df(\bx^{k+1})-\df(\bx^k))=\zero,\label{eq:up2_line0}\\		
			&\bx^{k+1}-\bx^*+cU^\T(\bq^{k+1}-\bq^*)-c\Lap(\df(\bx^{k+1})-\df(\bx^k))=\zero,\label{eq:up2_line1}\\
			&\bq^{k+1}=\bq^k+U\df(\bx^{k+1}),\label{eq:up2_line2}\\
			&\by^k=U^\T\bq^k,\label{eq:up2_line3}
		\end{align}
	\end{subequations}
	where $\bx^0$ is the same as in Algorithm 2, $\bq^{-1}=\zero$, and $\bq^0=U\df(\bx^0)$.
\end{lemma}
\an{\begin{proof}
		Relations~\eqref{eq:up2_line0}, \eqref{eq:up2_line2}, and~\eqref{eq:up2_line3} are obtained
		using the analysis that is similar to that for deriving the corresponding relations in Lemma~\ref{lemma:rec_rel_alg1}, so we omit their proofs. Relation~\eqref{eq:up2_line1} is obtained 
		from~\eqref{eq:up2_line0} by using $\br= \bx^* +U^\T\bq^*$ (see Lemma~\ref{lem:opc-res}).
	\end{proof}
}

Comparing the description of Algorithm 2 (Lemma \ref{lemma:rec_rel_alg2}) and the description of Algorithm 1 (Lemma \ref{lemma:rec_rel_alg1}), we can see that the only 
difference between these algorithms is in the scaling matrix of the ``proximal term''
$(\sdh(\bx^{k+1})-\sdh(\bx^k))$. In Algorithm 1 the coefficient matrix is $\beta I-c\Lap$, 
while in Algorithm 2 it is 
$-c\Lap$. 
We \an{refer to} the algorithm Mirror-EXTRA due to its relation to P-EXTRA and the fact that it features a gradient-based update, just as the EXTRA algorithm does for consensus optimization \cite{Shi2015}.

We start with a lemma that provides an important relation 
for the convergence analysis of Mirror-EXTRA.

\an{\begin{lemma}\label{lemma:contraction_lemma}
		Given that $a$ and $b$ are nonnegative, let sequences $\{\bv^k,\bu^k\}\in\R^{n\times p}\times \R^{n\times p}$ and a point $(\bv^*,\bu^*)\in\R^{n\times p}\times\R^{n\times p}$ be such that	
		\begin{equation}\label{eq:con_l2_1}
			b\|\bu^{k+1}-\bu^*\|_\Fro^2+a\|\bv^{k+1}-\bv^*\|_\Fro^2
			\leq b\|\bu^{k}-\bu^*\|_\Fro^2-b\|\bu^{k}-\bu^{k+1}\|_\Fro^2-b\|\bv^k-\bv^*\|_\Fro^2+b\|\bv^k-\bv^{k+1}\|_\Fro^2.
		\end{equation}
		Then, for any $\rho$ and $t$ such that
		\begin{numcases}{}
			a(1-t)\frac{1}{\rho}-b\leq at,\label{eq:con_l2_4_1}\\
			a(1-t)\frac{1}{1+\rho}-b>0,\label{eq:con_l2_4_2}
		\end{numcases}
		the sequence $\{b\|\bu^{k}-\bu^*\|_\Fro^2+ at\|\bv^k-\bv^*\|_\Fro^2\}$ is monotonically non-increasing
		and 
		$\sum_{k=0}^\infty \{b\|\bu^{k}-\bu^{k+1}\|_\Fro^2
		+\left(\frac{a(1-t)}{1+\rho}-b\right)\|\bv^k-\bv^{k+1}\|_\Fro^2\}<\infty$.
	\end{lemma}
}

\begin{proof}\an{
		We start by showing a relation valid for arbitrary vectors $x,y,z$ and for any $\rho>0$: 
		\begin{equation}\label{eq:rel0}
			\frac{1}{1+\rho}\|x-y\|_2^2 \le \frac{1}{\rho}\|x-z\|_2^2 + \|y-z\|_2^2.
		\end{equation}
		The preceding relation follows by noting that 
		$\|x-y\|_2^2\le  (1+ \beta)\|x-z\|_2^2+(1+\frac{1}{\beta}) \|y-z\|_2^2$ for any $\beta>0$. 
		Upon dividing both sides with $(1+\frac{1}{\beta})$ and letting $\rho=\frac{1}{\beta}$, 
		we obtain~\eqref{eq:rel0}.
		We note that the relation is valid for Frobenius norm as well, thus for arbitrary matrices $\bx,\by,\bz\in\R^{n\times p}$ and any $\rho>0$, we have
		\begin{equation}\label{eq:rel01}
			\frac{1}{1+\rho}\|\bx-\by\|_\Fro^2 \le \frac{1}{\rho}\|\bx-\bz\|_\Fro^2 + \|\by-\bz\|_\Fro^2.
		\end{equation}
	}
	
	\an{Now, by multiplying the relation~\eqref{eq:rel01} 
		with $a(1-t)$ where $t\in(0,1)$, we obtain
		\begin{equation*}
			\frac{a(1-t)}{1+\rho}\|\bx - \by\|_\Fro^2 \le \frac{a(1-t)}{\rho}\|\bx - \bz\|_\Fro^2 + a(1-t)\|\by-\bz\|_\Fro^2.
		\end{equation*}
		Letting $\bx=\bv^k$, $\by=\bv^{k+1}$,  and $\bz=\bv^*$, we have
		\begin{equation}\label{eq:rel02}
			\frac{a(1-t)}{1+\rho}\|\bv^k-\bv^{k+1}\|_\Fro^2 
			\le \frac{a(1-t)}{\rho}\|\bv^k - \bv^*\|_\Fro^2 + a(1-t)\|\bv^{k+1}-\bv^*\|_\Fro^2.
		\end{equation}
		Next, by adding relations~\eqref{eq:rel02} and~\eqref{eq:con_l2_1}, we find that
		\begin{equation}\label{eq:con_l2_3}
			\begin{array}{rcl}
				&&b\|\bu^{k+1}-\bu^*\|_\Fro^2
				+at\|\bv^{k+1}-\bv^*\|_\Fro^2+b\|\bu^{k}-\bu^{k+1}\|_\Fro^2
				+\left(\frac{a(1-t)}{1+\rho}-b\right)\|\bv^k-\bv^{k+1}\|_\Fro^2\\
				&\leq& b\|\bu^{k}-\bu^*\|_\Fro^2+\left(\frac{a(1-t)}{\rho}-b\right)\|\bv^k-\bv^*\|_\Fro^2
			\end{array}
		\end{equation}
		for arbitrary $t\in(0,1)$ and $\rho>0$. If $\rho$ and $t$ are such that
		the conditions in~\eqref{eq:con_l2_4_1} and~\eqref{eq:con_l2_4_2} 
		are satisfied, 
		then $\{b\|\bu^{k}-\bu^*\|_\Fro^2+at\|\bv^k-\bv^*\|_\Fro^2\}$ is monotonically non-increasing and
		$\sum_{k=0}^\infty
		\{b\|\bu^{k}-\bu^{k+1}\|_\Fro^2+\left(\frac{a(1-t)}{1+\rho}-b\right)\|\bv^k-\bv^{k+1}\|_\Fro^2\}<\infty$.
	}
\end{proof}

\begin{remark}[Parameters $t$ and $\rho$]\label{remark:par}
	Regarding the conditions of Lemma~\ref{lemma:contraction_lemma} in~\eqref{eq:con_l2_4_1} and~\eqref{eq:con_l2_4_2}, we note that
	if $a>3b>0$, then it can be verified that the conditions are met by choosing 
	$\rho=\frac{a-b}{2b}$ and any $t\in[\frac{b}{a},\frac{a-b}{2a}).$ 
\end{remark}

\begin{theorem}[Convergence of Mirror-EXTRA]\label{theorem:conv2}
	\an{Let $\Omega_i=\R^p$ for all $i$, and let Assumptions~\ref{assum:all}, \ref{ass:conn_and_prop}, and~\ref{ass:asss_for_geo1} be satisfied. Also, assume that problem~\eqref{eq:F2} has a solution, i.e., 
		$\bX^*\ne\emptyset$. If the parameter $c$ is chosen such that $0<c<\frac{1}{2L\laml{\Lap}}$,
		then the sequence $\{\bx^k\}$ converges to a point in $\bX^*$.}
\end{theorem}
\begin{proof}
	By the convexity and the gradient Lipschitz continuity of $\f$, we have \ws{for any $\bx^*\in\bX^*$,}
	\begin{equation}\label{eq:conv2_p1}
		\begin{array}{rcl}
			\frac{2}{\tL}\|\df(\bx^{k+1})-\df(\bx^*)\|_\Lap^2\leq\frac{2}{L}\|\df(\bx^{k+1})-\df(\bx^*)\|_\Fro^2\leq 2\langle\df(\bx^{k+1})-\df(\bx^*),\bx^{k+1}-\bx^*\rangle,
		\end{array}
	\end{equation}
	where $\tL=L\laml{\Lap}$. 
	\an{From Lemma~\ref{lemma:rec_rel_alg2} (cf.~\eqref{eq:up2_line1}) 
		we have an expression for $\bx^{k+1}-\bx^*$, which when
		substituted in~\eqref{eq:conv2_p1} yields}
	\begin{equation}\label{eq:conv2_p2}
		\begin{array}{rcl}
			&    &\frac{2}{\tL}\|\df(\bx^{k+1})-\df(\bx^*)\|_\Lap^2\\
			&\leq&2\langle\df(\bx^{k+1})-\df(\bx^*),cU^\T(\bq^*-\bq^{k+1})+c\Lap(\df(\bx^{k+1})-\df(\bx^{k}))\rangle\\
			&  = &2\langle\bq^{k+1}-\bq^k,c(\bq^*-\bq^{k+1})\rangle-2\langle\df(\bx^{k+1})-\df(\bx^*),c\Lap(\df(\bx^{k})-\df(\bx^{k+1}))\rangle\\
			&  = &c\|\bq^k-\bq^*\|_\Fro^2-c\|\bq^{k+1}-\bq^*\|_\Fro^2-c\|\bq^k-\bq^{k+1}\|_\Fro^2\\
			&   &-c\|\df(\bx^k)-\df(\bx^*)\|_\Lap^2+c\|\df(\bx^{k+1})-\df(\bx^*)\|_\Lap^2+c\|\df(\bx^k)-\df(\bx^{k+1})\|_\Lap^2,
		\end{array}
	\end{equation}
	\an{where the first equality follows from relation~\eqref{eq:up2_line2} of Lemma~\ref{lemma:rec_rel_alg2}
		and the optimality condition $U\df(\bx^*)=0$ 
		(cf.~Lemma~\ref{lem:opc-res} with $\widetilde\nabla \bh(\bx^*)=\nabla\f(\bx^*)$).
		Therefore,
		\begin{equation}\label{eq:conv2_p3}
			\begin{array}{rcl}
				&    &c\|\bq^{k+1}-\bq^*\|_\Fro^2+\left(\frac{2}{\tL}-c\right)\|\df(\bx^{k+1})-\df(\bx^*)\|_\Lap^2\\
				&\leq &c\|\bq^k-\bq^*\|_\Fro^2 -c\|\bq^k-\bq^{k+1}\|_\Fro^2-c\|\df(\bx^k)-\df(\bx^*)\|_\Lap^2+c\|\df(\bx^k)-\df(\bx^{k+1})\|_\Lap^2.
			\end{array}
		\end{equation}		
	}
	
	\an{We now apply Lemma \ref{lemma:contraction_lemma} with the following identification:
		in \eqref{eq:conv2_p3}, we set  $\{\bu^k,\bv^k\}\triangleq\{\bq^k,U\df(\bx^k)\}$, the point $\{\bu^*,\bv^*\}\triangleq\{\bq^*,U\df(\bx^*)\}$ (note that $\Lap=U^\T U$), and the quantities $b=c$, $a=\frac{2}{L'}-c$ and $\rho=\frac{a-b}{2b}$. 
		The condition $c\in\left(0,\frac{1}{2\tL}\right)$ is equivalent to $a>3b>0$ (see Remark~\ref{remark:par}).
		Thus, by
		applying  Lemma~\ref{lemma:contraction_lemma},
		we obtain that for any $c\in\left(0,\frac{1}{2\tL}\right)$ and 
		$t\in\left[\frac{c\tL}{2-c\tL},\frac{1-c\tL}{2-c\tL}\right)$, 
		\begin{equation}\nonumber\label{eq:conv2_p4_0}
			\text{the sequence }\left\{c\|\bq^k-\bq^*\|_\Fro^2+at\|\df(\bx^k)-\df(\bx^{*})\|_\Lap^2\right\}\text{ is monotonically non-increasing, and}
		\end{equation}
		\begin{equation}\label{eq:conv2_p4}
			\sum_{k=0}^\infty \left\{c\|\bq^k-\bq^{k+1}\|_\Fro^2+\left((2-c\tL)(1-t)c-c\right)\|\df(\bx^k)-\df(\bx^{k+1})\|_\Lap^2\right\}<\infty.
		\end{equation}
		This immediately implies the summability of the sequences 
		$\{\|\bq^k-\bq^{k+1}\|_\Fro^2\}$ and $\{\|\df(\bx^k)-\df(\bx^{k+1})\|_\Lap^2\}$. 
		In view of relations~\eqref{eq:up2_line2} and~\eqref{eq:up2_line0} of Lemma~\ref{lemma:rec_rel_alg2}, respectively,
		we have 
		\[\|\bq^k-\bq^{k+1}\|_\Fro^2=\|U\df(\bx^{k+1})\|_\Fro^2,\]
		\[\|\bx^{k+1}-\br+cU^\T\bq^{k+1}\|_\Fro^2=\|c\Lap(\df(\bx^k)-\df(\bx^{k+1}))\|_\Fro^2\leq c^2\laml{\Lap}\|\df(\bx^k)-\df(\bx^{k+1})\|_\Lap^2.\]
		Choosing $t=\frac{1}{2(2-c\tL)}$ in \eqref{eq:conv2_p4}, and using  the preceding two relations, 
		we obtain
		\[\sum_{k=0}^\infty\{c\|U\df(\bx^{k+1})\|_\Fro^2
		+\frac{0.5-cL\laml{\Lap}}{c\laml{\Lap}}\|\bx^{k+1}-\br+cU^\T\bq^{k+1}\|_\Fro^2\}<\infty.\]
		By following a nearly identical line of analysis as in the proof of Theorem~\ref{theorem:conv},
		from relations~\eqref{eq:conv_p4_line1} and~\eqref{eq:conv_p4_line2} onward, we can conclude that 
		$\{\bx^k\}$ converges to a point in the optimal set $\bX^*$.}
\end{proof}

\begin{remark}[Step size selection]\label{remark:step}
	The bound $c\leq\frac{1}{2L\laml{\Lap}}$ does not \an{necessarily imply that the step size $c$ selection
		requires any knowledge of the graph $\G$ structure. For example, 
		such a requirement can be avoided by using $\Lap=0.5(I-W)$,} 
	where $W$ is a symmetric stochastic matrix, in which case one may employ a step size 
	$c\leq\frac{1}{2L}$.
\end{remark}

\an{We next investigate convergence rate properties of Mirror-EXTRA,
	where we make use of the following result which is based on Lemma~\ref{lemma:rec_rel_alg2}.}

\begin{lemma}[Monotonic successive difference of Mirror-EXTRA]\label{lemma:mono2}
	\an{Let $\Omega_i=\R^p$ for all $i$, and let Assumptions~\ref{assum:all}, \ref{ass:conn_and_prop}, and~\ref{ass:asss_for_geo1} be satisfied. Then, for any $c\in\left(0,\frac{1}{2L\laml{\Lap}}\right)$ we have}
	$$
	c\|\D\bq^{k+1}\|_\Fro^2+(\frac{2}{L\laml{\Lap}}-3c)\|\D\df(\bx^{k+1})\|_\Lap^2\leq c\|\D\bq^k\|_\Fro^2+(\frac{2}{L\laml{\Lap}}-3c)\|\D\df(\bx^k)\|_\Lap^2,
	$$
	where $\D\bq^{k+1}\triangleq\bq^k-\bq^{k+1}$ and $\D\df(\bx^{k+1})\triangleq\D\df(\bx^k)-\D\df(\bx^{k+1})$.
\end{lemma}

\begin{proof}
	To simplify the notation, we also define 
	$\D\bx^{k+1}\triangleq\bx^k-\bx^{k+1}$. 
	By the convexity and the Lipschitz continuity of $\nabla \f$, we have
	\begin{equation}\label{eq:mono2_proof_1}
		\begin{array}{rcl}
			\langle\D\bx^{k+1},\D\df(\bx^{k+1})\rangle\geq\frac{1}{L}\|\D\df(\bx^{k+1})\|_\Fro^2\geq\frac{1}{\tL}\|\D\df(\bx^{k+1})\|_\Lap^2,
		\end{array}
	\end{equation}
	where $\tL=L\laml{\Lap}$. Now we use relation \eqref{eq:up2_line0} of Lemma~\ref{lemma:rec_rel_alg2};
	specifically, by taking the difference between \eqref{eq:up2_line0} at the $k$-th and
	$(k+1)$-th iteration we obtain
	\begin{equation}\label{eq:mono2_proof_2}
		\begin{array}{c}
			\D\bx^{k+1}+cU^\T\D\bq^{k+1}-c\Lap(\D\df(\bx^{k+1})-\D\df(\bx^k))=\zero.
		\end{array}
	\end{equation}
	\an{From~\eqref{eq:mono2_proof_2} we have an expression for $\D\bx^{k+1}$, which when substituted in
		relation~\eqref{eq:mono2_proof_1} yields}
	\begin{equation}\label{eq:mono2_proof_3}
		\begin{array}{rl}
			\langle-cU^\T\D\bq^{k+1}+c\Lap(\D\df(\bx^{k+1})-\D\df(\bx^k)),\D\df(\bx^{k+1})\rangle\geq\frac{1}{\tL}\|\D\df(\bx^{k+1})\|_\Lap^2.
		\end{array}
	\end{equation}
	\an{Relation \eqref{eq:up2_line2} of Lemma~\ref{lemma:rec_rel_alg2} implies 
		$\D\bq^{k+1}=-U\df(\bx^{k+1})$, which in turn gives
		\begin{equation}\label{eq:mono2_proof_4}
			\begin{array}{rcl}
				\D\bq^{k}-\D\bq^{k+1}=-U\D\df(\bx^{k+1}).
			\end{array}
		\end{equation}
		By substituting \eqref{eq:mono2_proof_4} into \eqref{eq:mono2_proof_3}, we can see that}
	\begin{equation}\label{eq:mono2_proof_5}
		\begin{array}{c}
			2c \langle \D\bq^{k+1},\D\bq^{k}-\D\bq^{k+1}\rangle
			+2c \langle \Lap\D\df(\bx^{k+1}),\D\df(\bx^{k+1})-\D\df(\bx^{k})\rangle
			\geq\frac{2}{\tL}\|\D\df(\bx^{k+1})\|_\Lap^2.
		\end{array}
	\end{equation}
	\an{Note that
		\[2\langle \D\bq^{k+1},\D\bq^{k}-\D\bq^{k+1}\rangle=\|\D\bq^{k}\|_\Fro^2 - \|\D\bq^{k+1}\|_\Fro^2 - 
		\|\D\bq^{k}-\D\bq^{k+1}\|_\Fro^2.\]
		Similarly, since $\Lap=U^\T U$, we have
		\[2 \langle \Lap\D\df(\bx^{k+1}),\D\df(\bx^{k+1})-\D\df(\bx^{k})\rangle
		=\|\D\df(\bx^{k+1})\|_\Lap^2 + \|\D\df(\bx^{k+1})-\D\df(\bx^{k})\|_\Lap^2 - \|\D\df(\bx^{k})\|_\Lap^2.\]
		By using the preceding two equalities in \eqref{eq:mono2_proof_5} and by reorganizing terms, 
		we obtain}
	\begin{equation*}
		\begin{array}{rl}
			&c\|\D\bq^{k+1}\|_\Fro^2+(\frac{2}{\tL}-c)\|\D\df(\bx^{k+1})\|_\Lap^2\\
			\leq&c\|\D\bq^k\|_\Fro^2-c\|\D\bq^k-\D\bq^{k+1}\|_\Fro^2+c\|\D\df(\bx^k)-\D\df(\bx^{k+1})\|_\Lap^2
			-c\|\D\df(\bx^k)\|_\Lap^2\\
			\leq& c\|\D\bq^k\|_\Fro^2+c\|\D\df(\bx^k)\|_\Lap^2+2c\|\D\df(\bx^{k+1})\|_\Lap^2,
		\end{array}
	\end{equation*}
	\an{where the last inequality follows from $\|\D\df(\bx^k)-\D\df(\bx^{k+1})\|_\Lap^2\le 
		2\|\D\df(\bx^k)\|_\Lap^2 +2\|\D\df(\bx^{k+1})\|_\Lap^2.$
		Therefore,
		\begin{equation*}
			\begin{array}{rcl}
				c\|\D\bq^{k+1}\|_\Fro^2+(\frac{2}{\tL}-3c)\|\D\df(\bx^{k+1})\|_\Lap^2
				&\leq&c\|\D\bq^k\|_\Fro^2+c\|\D\df(\bx^k)\|_\Lap^2\\
				&\leq&c\|\D\bq^k\|_\Fro^2+(\frac{2}{\tL}-3c)\|\D\df(\bx^k)\|_\Lap^2,
			\end{array}
		\end{equation*}
		where in the last inequality we use $c\le \frac{2}{\tL}-3c$, which holds by the assumption that 
		$c\in\left(0,\frac{1}{2\tL}\right)$.} 
\end{proof}

\an{We have the following basic rate result for the iterates of Algorithm 2, 
	when the objective function $\f$ is convex and has Lipschitz continuous gradients. 
	The result follows directly from Theorem~\ref{theorem:conv2}, Lemma~\ref{lemma:mono2}, and 
	Proposition~\ref{prop:o_1_k}.}
\begin{theorem}[Sublinear rate of Mirror-EXTRA]\label{theorem:o_1_k_2}
	\an{Under the assumptions of Theorem~\ref{theorem:conv2},
		along the iterates of Algorithm 2,}  
	the first-order optimality residual decays to $0$ at an $o\left(\frac{1}{k}\right)$ rate, i.e.,
	\[\|U\df(\bx^{k+1})\|_\Fro^2=o\left(\frac{1}{k}\right), \qquad \|\bx^{k+1}-\br+cU^\T\bq^{k+1}\|_{\Fro}^2=o\left(\frac{1}{k}\right).\]
\end{theorem}

We next show a linear convergence of Mirror-EXTRA under 
additional strong convexity assumption on the function $\f$. 
To simplify the analysis, {\it we will assume that $\laml{\Lap}\leq1$},which holds for example 
when $\Lap=\Laps/\laml{\Laps}$ or $\Lap=0.5(I-W)$ for some symmetric stochastic matrix $W$ compatible with the graph $\G$.

\begin{theorem}[Linear rate of Mirror-EXTRA]\label{theorem:geo2}
	\an{Let $\Omega_i=\R^p$ for all $i$, and 
		let Assumptions~\ref{ass:conn_and_prop}, \ref{ass:asss_for_geo1} and~\ref{ass:asss_for_geo2} hold.
		Furthermore, suppose that $\Lap$ is such that $\laml{\Lap}\leq1$ and the step size $c$ satisfies 
		$c\in(0,\frac{1}{2L})$. Then, the iterate sequence $\{\bx^k\}$ 
		generated by Algorithm 2 converges to the optimal solution $\bx^*$ of problem~\eqref{eq:basic}
		at an R-linear rate, i.e.,
		\[\|\bx^{k+1}-\bx^*\|_\Fro^2=O\left(\frac{1}{\left(1+\delta\right)^k}\right),\]
		where $\delta>0$ is such that for some $\gamma>0$, 
		\[ \delta\leq\min\left\{
		\frac{2\mu Lc\lams{\Lap}}{(\mu+L)(1+\gamma)},\frac{c^2L^2\lams{\Lap}}{1+\gamma+cL(2-3cL)\lams{\Lap}+2c^2L^2(1+\frac{1}{\gamma})},\frac{(1-2cL)\lams{\Lap}}{cL(1+\frac{1}{\gamma})}\right\}.
		\end{equation*}
	}
	\end{theorem}
	\begin{proof}
	\an{By the strong convexity and the gradient Lipschitz continuity of the function $\f$, and the assumption that 
		$\laml{\Lap}\leq1$,}
	we have
	\begin{equation}\label{eq:geo2_p1}
	\begin{array}{rcl}
	\frac{2\mu L}{\mu+L}\|\bx^{k+1}-\bx^*\|_\Fro^2+\frac{2}{\mu+L}\|\df(\bx^{k+1})-\df(\bx^*)\|_\Lap^2&\leq&2\langle\df(\bx^{k+1})-\df(\bx^*),\bx^{k+1}-\bx^*\rangle.
	\end{array}
	\end{equation}	
	\an{
		From relation \eqref{eq:up2_line1} of Lemma~\ref{lemma:rec_rel_alg2} it follows that
		\begin{align}\label{eq:aux0}
		\bx^{k+1}-\bx^*=cU^\T(\bq^*-\bq^{k+1})+c\Lap(\df(\bx^{k+1})-\df(\bx^k)),\end{align}
		which when substituted in~\eqref{eq:geo2_p1} yields
		\begin{equation}\label{eq:geo2_p1-1}
		\begin{array}{rcl}
		\frac{2\mu L}{\mu+L}\|\bx^{k+1}-\bx^*\|_\Fro^2
		+\frac{2}{\mu+L}\|\df(\bx^{k+1})-\df(\bx^*)\|_\Lap^2
		&\leq&2c\langle U(\df(\bx^{k+1})-\df(\bx^*)), \bq^*-\bq^{k+1}\rangle\cr
		&&+2c \langle\df(\bx^{k+1})-\df(\bx^*),\Lap(\df(\bx^{k+1})-\df(\bx^k))\rangle.
		\end{array}
		\end{equation}	
		Since $\Lap=U^\T U$, we have
		\begin{align}\label{eq:aux1}
		2 \langle \df(\bx^{k+1})-\df(\bx^*),\Lap(\df(\bx^{k+1})-\df(\bx^{k}))\rangle
		=&\|\df(\bx^{k+1})-\df(\bx^*)\|_\Lap^2 + \|\df(\bx^{k+1})-\df(\bx^{k})\|_\Lap^2\cr
		& - \|\df(\bx^{k})-\df(\bx^*)\|_\Lap^2.\quad\end{align}
		By the optimality condition $U\df(\bx^*)=0$, it follows that 
		\begin{align}\label{eq:aux1-1}
		U(\df(\bx^{k+1})-\df(\bx^*))=U\df(\bx^{k+1})=\bq^{k+1}-\bq^k,\
		\end{align}
		where the last equality follows from relation \eqref{eq:up2_line2} of Lemma~\ref{lemma:rec_rel_alg2}. 
		Therefore,
		\begin{align}\label{eq:aux2}
		2\langle U(\df(\bx^{k+1})-\df(\bx^*)), \bq^*-\bq^{k+1}\rangle
		&=2\langle \bq^{k+1}-\bq^k, \bq^*-\bq^{k+1}\rangle\cr
		&=\|\bq^k- \bq^*\|_\Fro^2 -\|\bq^{k+1}-\bq^k\|_\Fro^2 - \|\bq^{k+1}-\bq^*\|_\Fro^2.
		\end{align}
		Upon substituting relations~\eqref{eq:aux1} and~\eqref{eq:aux2} in~\eqref{eq:geo2_p1-1}, after re-arranging the terms,
		we obtain
		\begin{equation*}
		\begin{array}{rcl}
		&    &\frac{2\mu L}{\mu+L}\|\bx^{k+1}-\bx^*\|_\Fro^2+\frac{2}{\mu+L}\|\df(\bx^{k+1})-\df(\bx^*)\|_\Lap^2
		+c\|\bq^k-\bq^{k+1}\|_\Fro^2-c\|\df(\bx^k)-\df(\bx^{k+1})\|_\Lap^2\\
		&\leq&c\|\bq^k-\bq^*\|_\Fro^2-c\|\bq^{k+1}-\bq^*\|_\Fro^2
		-c\|\df(\bx^k)-\df(\bx^*)\|_\Lap^2+c\|\df(\bx^{k+1})-\df(\bx^*)\|_\Lap^2.
		\end{array}
		\end{equation*}
		Now, we use~\eqref{eq:aux1-1} and we add 
		$(\frac{2}{L}-2c)\|\df(\bx^k)-\df(\bx^*)\|_\Lap^2-(\frac{2}{L}-2c)\|\df(\bx^{k+1})-\df(\bx^*)\|_\Lap^2$
		to both sides of the preceding inequality, which gives
		\begin{equation*}
		\begin{array}{rcl}
		&&\frac{2\mu L}{\mu+L}\|\bx^{k+1}-\bx^*\|_\Fro^2
		+(\frac{2}{\mu+L}+c-\frac{2}{L}+2c)\|\df(\bx^{k+1})-\df(\bx^*)\|_\Lap^2\cr
		&&+(\frac{2}{L}-2c)\|\df(\bx^k)-\df(\bx^*)\|_\Lap^2-c\|\df(\bx^k)-\df(\bx^{k+1})\|_\Lap^2\cr
		&&\le 
		c\|\bq^k-\bq^*\|_\Fro^2-c\|\bq^{k+1}-\bq^*\|_\Fro^2+(\frac{2}{L}-3c)\|\df(\bx^k)-\df(\bx^*)\|_\Lap^2-(\frac{2}{L}-3c)\|\df(\bx^{k+1})-\df(\bx^*)\|_\Lap^2.	
		\end{array}
		\end{equation*}
	}
	
	In view of the preceding relation, in order to prove the linear convergence, it suffices to show that 
	for some $\delta>0$ the following relation holds for all $k\ge1$,
	\begin{equation}\label{eq:geo2_p3}
	\begin{array}{rcl}
	&&\delta c\|\bq^{k+1}-\bq^*\|_\Fro^2+\delta(\frac{2}{L}-3c)\|\df(\bx^{k+1})-\df(\bx^*)\|_\Lap^2\cr
	&\le&
	\frac{2\mu L}{\mu+L}\|\bx^{k+1}-\bx^*\|_\Fro^2
	+(\frac{2}{\mu+L}+3c-\frac{2}{L})\|\df(\bx^{k+1})-\df(\bx^*)\|_\Lap^2\cr
	&&+(\frac{2}{L}-2c)\|\df(\bx^k)-\df(\bx^*)\|_\Lap^2-c\|\df(\bx^k)-\df(\bx^{k+1})\|_\Lap^2.
	\end{array}
	\end{equation}
	\an{To see this, note that assuming that relation~\eqref{eq:geo2_p3} is valid, we will have
		\begin{equation*}
		\begin{array}{rcl}
		&&\delta c\|\bq^{k+1}-\bq^*\|_\Fro^2+\delta(\frac{2}{L}-3c)\|\df(\bx^{k+1})-\df(\bx^*)\|_\Lap^2\cr
		&\le& c\|\bq^k-\bq^*\|_\Fro^2-c\|\bq^{k+1}-\bq^*\|_\Fro^2
		+(\frac{2}{L}-3c)\|\df(\bx^k)-\df(\bx^*)\|_\Lap^2-(\frac{2}{L}-3c)\|\df(\bx^{k+1})-\df(\bx^*)\|_\Lap^2,
		\end{array}
		\end{equation*}
		or equivalently
		$$c\|\bq^{k+1}-\bq^*\|_\Fro^2
		+(\frac{2}{L}-3c)\|\df(\bx^{k+1})-\df(\bx^*)\|_\Lap^2
		\leq\frac{1}{1+\delta}\left(c\|\bq^{k}-\bq^*\|_\Fro^2+(\frac{2}{L}-3c)\|\df(\bx^{k})-\df(\bx^*)\|_\Lap^2\right),$$
		which implies that that both sequences $\{\|\bq^{k+1}-\bq^*\|_\Fro^2\}$ and 
		$\{\|\df(\bx^{k+1})-\df(\bx^*)\|_\Lap^2\}$ converge to zero at the R-linear rate 
		$O\left(\frac{1}{1+\delta}\right)$.
		Moreover, by equality~\eqref{eq:aux0}, it will follow that
		\begin{equation}\label{eq:geo2_p8}
		\begin{array}{rcl}
		\|\bx^{k+1}-\bx^*\|_\Fro^2
		&=& O\left(\frac{1}{\left(1+\delta\right)^k}\right),
		\end{array}
		\end{equation}
		showing that the iterate sequence $\{\bx^k\}$ converges to the optimal solution $\bx^*$ at an R-linear rate.
	}
	
	\an{The rest of the proof is concerned with finding sufficient conditions on $\delta$ for
		relation~\eqref{eq:geo2_p3} to hold. We start by noticing that,
		since $\Lap=U^\T U$ and the matrix $U\in\R^{(n-1)\times n}$ is full row rank 
		(see Assumption~\ref{ass:conn_and_prop}), and $\laml{UU^\T}=\laml{\Lap}\leq1$, from \eqref{eq:up2_line1} we have for any $\gamma>0$,
		\begin{equation*}
		\begin{array}{rcl}
		c^2\lams{\Lap}\|\bq^{k+1}-\bq^*\|_\Fro^2
		&\leq&
		(1+\gamma)\|\bx^{k+1}-\bx^*\|_\Fro^2
		+\left(1+\frac{1}{\gamma}\right)c^2\|\df(\bx^{k+1})-\df(\bx^k)\|_\Lap^2;
		\end{array}
		\end{equation*}
		furthermore, it always holds that $\|\df(\bx^k)-\df(\bx^{k+1})\|_\Lap^2\leq2\|\df(\bx^k)-\df(\bx^{*})\|_\Lap^2+2\|\df(\bx^{k+1})-\df(\bx^{*})\|_\Lap^2$.
		From the preceding two inequalities it follows that relation~\eqref{eq:geo2_p3} is valid 
		as long as the following relation holds} 
	\begin{equation*}
	\begin{array}{rcl}
	&&\frac{\delta(1+\gamma)}{c\lams{\Lap}}\|\bx^{k+1}-\bx^*\|_\Fro^2+2\left(\frac{\delta c(1+\frac{1}{\gamma})}{\lams{\Lap}}+c\right)\|\df(\bx^k)-\df(\bx^*)\|_\Lap^2\cr
	&\le&
	\frac{2\mu L}{\mu+L}\|\bx^{k+1}-\bx^*\|_\Fro^2+(\frac{-2\mu}{(\mu+L)L}+c-\delta(\frac{2}{L}-3c)-\frac{2\delta c(1+\frac{1}{\gamma})}{\lams{\Lap}})\|\df(\bx^{k+1})-\df(\bx^*)\|_\Lap^2\\
	&&+(\frac{2}{L}-2c)\|\df(\bx^k)-\df(\bx^*)\|_\Lap^2,
	\end{array}
	\end{equation*}
	or equivalently
	\begin{equation}\label{eq:geo2_p6}
	\begin{array}{rcl}
	&&\left(\frac{2\mu}{(\mu+L)L}-c+\delta(\frac{2}{L}-3c)+\frac{2\delta c(1+\frac{1}{\gamma})}{\lams{\Lap}}\right)\|\df(\bx^{k+1})-\df(\bx^*)\|_\Lap^2+2\left(\frac{\delta c(1+\frac{1}{\gamma})}{\lams{\Lap}}+c\right)\|\df(\bx^k)-\df(\bx^*)\|_\Lap^2\cr
	&\le &\left(\frac{2\mu L}{\mu+L}-\frac{\delta(1+\gamma)}{c\lams{\Lap}}\right)\|\bx^{k+1}-\bx^*\|_\Fro^2+(\frac{2}{L}-2c)\|\df(\bx^k)-\df(\bx^*)\|_\Lap^2.
	\end{array}
	\end{equation}
	
	\an{We next further examine some sufficient relations for \eqref{eq:geo2_p6} to be valid. 
		In particular, by the assumption that $\laml{\Lap}\le 1$ and by the Lipschitz continuity of $\df$,
		we have
		\[\frac{1}{L^2}\|\df(\bx^{k+1)}-\df(\bx^*)\|_\Lap^2\le \frac{1}{L^2}\|\df(\bx^{k+1)}-\df(\bx^*)\|_\Fro^2
		\le
		\|\bx^{k+1}-\bx^*\|_\Fro^2.\]
		Using the preceding inequality, when 
		$\frac{2\mu L}{\mu+L}-\frac{\delta(1+\gamma)}{c\lams{\Lap}}\ge0$, 
		relation~\eqref{eq:geo2_p6} will hold if the following relation holds:
		\begin{equation*}
			\begin{array}{rcl}
				&&\left(\frac{2\mu}{(\mu+L)L}-c+\delta(\frac{2}{L}-3c)+\frac{2\delta c(1+\frac{1}{\gamma})}{\lams{\Lap}}\right)\|\df(\bx^{k+1})-\df(\bx^*)\|_\Lap^2+2\left(\frac{\delta c(1+\frac{1}{\gamma})}{\lams{\Lap}}+c\right)\|\df(\bx^k)-\df(\bx^*)\|_\Lap^2\cr
				&\le &\frac{1}{L^2}\left(\frac{2\mu L}{\mu+L}-\frac{\delta(1+\gamma)}{c\lams{\Lap}}\right)
				\|\df(\bx^{k+1})-\df(\bx^*)\|_\Fro^2+(\frac{2}{L}-2c)\|\df(\bx^k)-\df(\bx^*)\|_\Lap^2.
			\end{array}
		\end{equation*}
		The preceding relation will hold,  as long as $\delta>0$ is small enough so that, 
		for some $\gamma>0$, we have}
	\begin{subequations}\label{eq:geo2_p7}
		\begin{numcases}{}
			\frac{2\mu L}{\mu+L}-\frac{\delta(1+\gamma)}{c\lams{\Lap}}\geq 0,\label{eq:geo2_p7_l0}\\
			\frac{2\mu L}{\mu+L}-\frac{\delta(1+\gamma)}{c\lams{\Lap}}\geq L^2\left(\frac{2\mu}{(\mu+L)L}-c+\delta(\frac{2}{L}-3c)+\frac{2\delta c(1+\frac{1}{\gamma})}{\lams{\Lap}}\right),\label{eq:geo2_p7_l1}\\
			\frac{2}{L}-2c\geq 2\left(\frac{\delta c(1+\frac{1}{\gamma})}{\lams{\Lap}}+c\right).\label{eq:geo2_p7_l2}
		\end{numcases}
	\end{subequations}
	\an{Given a $\gamma>0$, to satisfy the conditions in~\eqref{eq:geo2_p7},
		one can choose $\delta>0$ so that
		\begin{equation*}
			\delta\leq\min\left\{
			\frac{2\mu Lc\lams{\Lap}}{(\mu+L)(1+\gamma)},\frac{c^2L^2\lams{\Lap}}{1+\gamma+cL(2-3cL)\lams{\Lap}+2c^2L^2(1+\frac{1}{\gamma})},\frac{(1-2cL)\lams{\Lap}}{cL(1+\frac{1}{\gamma})}\right\}.
		\end{equation*}
	}
\end{proof}

\an{As a consequence of Theorem~\ref{theorem:geo2},
	we have the following corollary regarding the scalability of the Mirror-EXTRA.}

\begin{corollary}[Scalability of Mirror-EXTRA] 
	\an{Under the conditions of Theorem~\ref{theorem:geo2},} for Algorithm 2 to reach $\varepsilon$-accuracy, the number of iterations needed is \an{of the order} $O\left(\kappa_{\Lap}\kappa_\f\ln\left(\frac{1}{\varepsilon}\right)\right)$, where $\kappa_\f=\frac{L}{\mu}$ is the condition number of the objective function $\f$ 
	and $\kappa_\Lap=\frac{1}{\lams{\Lap}}$ is the condition number of the graph. 
\end{corollary}

\begin{proof}  	
	\an{From the bound on $\delta$ in Theorem~\ref{theorem:geo2} 
		we see that, to reach $\varepsilon$-accuracy,
		the number of iterations needed is (the factor $\ln\left(\frac{1}{\varepsilon}\right)$ is omitted for simplicity)}
	\begin{equation*}
		\begin{array}{rcl}
			&&O\left(\frac{(\mu+L)(1+\gamma)}{2\mu Lc\lams{\Lap}}\right)+O\left(\frac{1+\gamma+cL(2-3cL)\lams{\Lap}+2c^2L^2(1+\frac{1}{\gamma})}{c^2L^2\lams{\Lap}}\right)+O\left(\frac{cL(1+\frac{1}{\gamma})}{(1-2cL)\lams{\Lap}}\right)\\
			&=&O\left(\frac{1+\gamma}{\mu c\lams{\Lap}}\right)+O\left(\frac{1+\gamma}{c^2L^2\lams{\Lap}}\right)+O\left(\frac{1+\frac{1}{\gamma}}{\lams{\Lap}}\right)+O\left(\frac{1+\frac{1}{\gamma}}{\left(\frac{1}{cL}-2\right)\lams{\Lap}}\right).
		\end{array}
	\end{equation*}
	\an{Given $\gamma>0$ and $c\in(0,\frac{1}{2L})$,  the preceding relation
		can be further reduced to the order of $O\left(\frac{L}{\mu}\frac{1}{\lams{\Lap}}\right)$. 
		Hence, the final complexity is of the order 
		$O\left(\kappa_{\Lap}\kappa_\f\ln\left(\frac{1}{\varepsilon}\right)\right)$.} 
\end{proof}
\begin{remark}[Scalability]\an{
		The complexity of the Mirror-EXTRA is slightly worse than the complexity of Mirror-P-EXTRA, which is 
		$O\left((\kappa_{\Lap}+\sqrt{\kappa_{\Lap}\kappa_\f})\ln\left(\frac{1}{\varepsilon}\right)\right)$.
		The Mirror-EXTRA has a lower per-iteration cost 
		at the expense of less favorable scalability,  and $O\left(\kappa_{\Lap}\kappa_\f\ln\left(\frac{1}{\varepsilon}\right)\right)$ scalability of Mirror-EXTRA also coincides with that of the algorithm proposed 
		in reference~\cite{doan2017distributed}.}
\end{remark}

In the following subsections, we will first provide a gradient projection algorithm which can be understood as a hybrid of Algorithm 1 and Algorithm 2. Then we will conduct two case studies to show how the methodology of using the “mirror relation” can help to design more powerful distributed resource allocation algorithms based on existing consensus optimization algorithms.

\subsection{\cred{A projection gradient-based algorithm: Mirror-PG-EXTRA}}
Since Algorithm 1 has to solve a constrained optimization problem per iteration over each agent (which can be costly) while Algorithm 2 cannot be applied to constrained problems, we consider another algorithm that uses a gradient-projection update which can be understood as a hybrid of Algorithms 1 and 2. We name the newly constructed algorithm Mirror-PG-EXTRA (see Algorithm 3) after a previous algorithm for consensus optimization, PG-EXTRA \cite{Shi2015_2}.
\smallskip
\begin{center}
	{\textbf{Algorithm 3: Mirror-PG-EXTRA}}
	
\smallskip
\begin{tabular}{l}
		\hline
		\emph{  } Each agent chooses the same $c\in(0,\frac{1}{2L\laml{\Lap}})$;\\
		\emph{  } Each agent 
		$i$ chooses $\beta_i$ such that $B\succcurlyeq c\Lap$;\\
		\emph{  } $\forall i$, initialize with $x_i^0\in\Omega_i$, $s_i^0=0$, and $y_i^{-1}=0$;\\
		\emph{  } Each agent $i$ \textbf{for} $k=0,1,\ldots$ \textbf{do}\\
		\qquad$y_i^{k}=y_i^{k-1}+\sum_{j\in\Ni\bigcup\{i\}}\text{\L}_{ij}(\df_j(x_j^k)+s_j^{k})$;\\
		\qquad$x_i^{k+1}={\proj_{\Omega_i}}\left\{r_i-2cy_i^{k}+cy_i^{k-1}+\beta_i s_i^k\right\}$;\\
		\qquad$s_i^{k+1}=s_i^k-\frac{1}{\beta_i}(x_i^{k+1}-r_i+2cy_i^k-cy_i^{k-1})$;\\
		\emph{  } \textbf{end}\\
		\hline
	\end{tabular}
\end{center}
In Algorithm 3, the matrix $B$ is a diagonal matrix with entries $B_{ii}=\beta_i$ on its diagonal, while
$\proj_{\Omega_i}[\cdot]$ is the projection operator on the set $\Omega_i$ with respect to the Euclidean norm. In the absence of the per-agent-constraints, Algorithm 3 degenerates to Algorithm 2 \cred{(this can be verified by writing out the recursive relations of $\{\bx^k\}$, $\{\by^k\}$, and $\{\bs^k\}$ in Algorithm 3 and eliminating the sequence $\{\bs^k\}$ in the relations)}. We do not formally establish the convergence or the convergence rates for Algorithm 3, though we believe 
it has $1/k$ convergence rate under convexity and smoothness assumptions. An immediate guess on the convergence property of Algorithm 3 is that it will be slower than Algorithm 1 in terms of the number of iterations needed to reach a given accuracy. We will evaluate it numerically in our simulations (see Section \ref{sec:numer}).

\section{\cred{Improving the rates by Nesterov's acceleration}}

Suppose that the per-agent constraints are absent, i.e., $\Omg_i=\R^p$ for all $i$, there is a easy way to apply Nesterov's accelerated gradient method to achieve optimal convergence rates for solving the resource allocation problem \eqref{eq:F2} over fixed connected undirected graphs. To expose this fact, let us define \[ \bJ(\bz) = \f(\br+ U^\T\bz), \] where $U\in\R^{(n-1)\times n}$ is the matrix defined in Assumption \ref{ass:conn_and_prop}, i.e., a matrix that satisfies $\Lap=U^\T U$. Since the underlying graph is connected, we have that the columns of $U^\T$ span ${\bf 1}^{\perp}$ (because for any matrix $A$, we have that $\spa{A^T} = (\nul{A})^{\perp}$), and therefore the problem we want to solve is equivalent to unconstrained minimization of $\bJ(\bz)$. This can be seen by noticing the fact that minimizing $\f(\bx)$ over $\bx\in\R^{n\times p}$ and $\bz\in\R^{(n-1)\times p}$ subject to $\bx=\br+U^\T\bz$ will give us the optimal solution $(\bx^*,\bz^*)$ which actually includes a minimizer $\bz^*$ of the function $\bJ(\bz)$ and a solution $\bx^*$ to the resource allocation problem. In addition, $\bx^*$ can be recovered by $\bx^*=\br+U^\T\bz^*$ once $\bz^*$ is retrieved.

To make the discussion concise, without loss of generality, let us choose $\Lap$ such that $\|U\|_2=\sigmax{U}=1$. This can be done by choosing, for example, $\Lap=0.5(I-W)$ where $W$ is a symmetric doubly stochastic matrix that is compatible with the graph (see Subsection \ref{sec:notation} and Corollary \ref{cor:two}). Also suppose that Assumption \ref{ass:asss_for_geo2} (Lipscthiz continuity of $\df(\cdot)$) holds and $L=\max_i L_i$ is the Lipschitz constant of $\df(\cdot)$, then the gradient mapping of $\bJ$ is given by \[\dbJ(\bz) = U\df(\br + U^\T \bz) \] and is $\dbJ(\cdot)$ is $L$-Lipschitz as well.

Next we will show that an equivalent form of the Nesterov's accelerated gradient method on finding the minimizer of $\bJ(\bz)$ can be implemented in decentralized way and clearly this can be utilized to obtain $\bx^*$. The recursion of the Nesterov accelerated gradient method (see \cite{nesterov2013introductory,bubeck2015convex} for the introduction of the algorithm and its analysis) for minimizing $\bJ(\bz)$ is 
\begin{subequations}\label{eq:Nesterov}
	\begin{align}
	&\by^{k+1} =\bz^k - \alpha \dbJ(\bz^k), \\
	&\bz^{k+1} = \by^{k+1} + \beta_k(\by^{k+1}-\by^{k}),
	\end{align}
\end{subequations}
initialized with an arbitrary $\bz^0\in\R^{(n-1)\times p}$ and $\by^0=\bz^0$, where $\{\by^k\}$ is an auxiliary variable sequence and $\{\beta_k\}$ is scalar sequence of parameters that controls the extrapolation and should be chosen based on the functional conditions (whether $\bJ(\bz)$ is strongly convex in $\bz$ or not); the scalar $\alpha$ is a positive constant step size. Since the operations $U$ and $U^\T$ in the computation of $\dbJ(\bz^k)$ are not immediately implementable in finite rounds of communication, we cannot perform \eqref{eq:Nesterov} directly. However, we can substitute variables in a way to still find a sequence we need that converges to the optimal solution $\bx^*$.

For any $\bz^k$ and $\by^k$, let us define \[\bx^k = \br + U^\T \bz^k \text{ and } \bv^k = \br + U^\T \by^k.\] Then from \eqref{eq:Nesterov}, we can obtain the following recursions of $\bx^k$ and $\bv^k$ which inherit the convergence properties of \eqref{eq:Nesterov}:
\begin{subequations}\label{eq:Nesterov2}
	\begin{align}
	&\bv^{k+1} =\bx^k - \alpha \Lap\df(\bx^k), \\
	&\bx^{k+1} = \bv^{k+1} + \beta_k(\bv^{k+1}-\bv^{k})
	\end{align}
\end{subequations}
with initialization $\bv^0=\bx^0=\br+\Lap\bc$ with an arbitrary $\bc\in\R^{n\times p}$ (can be chosen distributedly).

The algorithm given in \eqref{eq:Nesterov2} can be implemented in a fully decentralized fashion. We thus can apply Nesterov's accelerated gradient method for the resource allocation problem over a network and obtain improved convergence rates. Specifically, we have the following two propositions for the claims of rates. Proofs are omitted due to the fact they are direct applications of Nesterov's method.

\begin{proposition}\label{prop:acc1}
	Suppose that Assumption \ref{ass:asss_for_geo1} holds and $\Lap$ is chosen such that $\sigmax{\Lap}=1$. If we choose $\alpha=\frac{1}{L}$ and $\beta_k=\frac{k}{k+3}$, with the algorithm given in \eqref{eq:Nesterov2}, it is guaranteed that \[\f(\bv^k)-\f(\bx^*)=O\left(\frac{L}{\lams{\Lap}}\frac{\|\bx^0-\bx^*\|_\Fro^2}{k^2}\right).\]	
\end{proposition}

To explore the geometric convergence, it is expected that one assumes strong convexity. Suppose that Assumption \ref{ass:asss_for_geo2} holds, then for any $\ba$ and $\bb$ with proper dimensions,
\begin{align}
&\langle\dbJ(\ba)-\dbJ(\bb),\ba-\bb\rangle\\
=&\langle U\df(\br + U^\T \ba)-U\df(\br + U^\T \bb),\ba-\bb\rangle\\
=&\langle \df(\br + U^\T \ba)-\df(\br + U^\T \bb), (\br+U^\T\ba)- (\br+U^\T\bb)\rangle\\
\geq&\mu\|U^\T(\ba-\bb)\|_\Fro^2\\
\geq&\mu\lams{\Lap}\|\ba-\bb\|_\Fro^2
\end{align}

\begin{proposition}\label{prop:acc2}
	Suppose that Assumptions \ref{ass:asss_for_geo1} and \ref{ass:asss_for_geo2} hold and $\Lap$ is chosen such that $\sigmax{\Lap}=1$. If we choose $\alpha=\frac{1}{L}$ and $\beta_k=\beta=\frac{\sqrt{Q}-1}{\sqrt{Q}+1}$ where $Q=\frac{L}{\mu\lams{\Lap}}$, then with the algorithm given in \eqref{eq:Nesterov2}, it is guaranteed that \[\f(\bv^k)-\f(\bx^*)=O\left(\left(\mu+\frac{L}{\lams{\Lap}}\right)\|\bx^k-\bx^*\|_\Fro^2 e^{-\frac{k}{\sqrt{Q}}}\right)\]
	where $e$ is the base of the natural logarithm. In other words, to reach $\varepsilon$-accuracy, the number of iterations needed is in the order of $O\left(\sqrt{\kappa_{\f}\kappa_{\Lap}}\ln(\varepsilon^{-1})\right)$ where the $\kappa_{\f}=L/\mu$ is the condition number of the objective and $\kappa_{\Lap}=1/\lams{\Lap}$ is the condition number of the graph. Furthermore, when $\Lap=0.5(I-W)$ where $W$ is the lazy Metropolis matrix (see also Corollary \ref{cor:two}), the number of iterations needed is of the order of $O\left(n\kappa_{\f}^{0.5}\ln(\varepsilon^{-1})\right)$ which is linear in $n$.
\end{proposition}

Although using the Nesterov's acceleration gives better dependency in $\kappa_{\f}$ and $\kappa_{\Lap}$, the needed algorithmic parameter $Q$ involves network wide information. When $W$ is the lazy Metropolis matrix, we can, to some extent, address this issue by relaxing the parameters we use. Instead of using $\lams{\Lap}$ or $n$, let us assume that all agents in the network know a common parameter $\hat{n}$ which is an upper bound of the network size $n$. We state the result in the following corollary.
\begin{corollary}
Suppose that Assumptions \ref{ass:asss_for_geo1} and \ref{ass:asss_for_geo2} hold and $\Lap=0.5(I-W)$ where $W$ is the lazy Metropolis matrix. If we choose $\alpha=\frac{1}{L}$ and $\beta_k=\beta=\frac{\sqrt{Q}-1}{\sqrt{Q}+1}$ where $Q=71\hat{n}^2L/\mu$, then in $O\left(\hat{n}\kappa_{\f}^{0.5}\ln(\varepsilon^{-1})\right)$ iterations, the algorithm will give $\f(\bv^k)-\f(\bx^*)\leq\varepsilon$.
\end{corollary}

The above discussion only applies to problems with smooth objectives assuming no local constraints. Adapting Nesterov's accelerated proximal gradient method for the resource allocation problem with local constraints will be one of our future directions.

\section{\cred{Extensions: Using the Mirror to Conquer More Complicated Scenarios}}\label{sec:ext}
In this section, we will conduct two case studies to show how our methodology can help to design more powerful distributed resource allocation algorithms based on existing consensus optimization algorithms. 
\subsection{Working over time-varying directed graphs: Mirror-Push-DIGing}

Before introducing the so called Mirror-Push-DIGing algorithm for solving the decentralized resource allocation problem, let us review the Push-DIGing algorithm which is proposed in reference \cite{Nedich2016} for solving the decentralized consensus optimization over time-varying directed graphs with geometric convergence guarantees. The procedure of Push-DIGing is as follows:

\begin{center}
	
	\begin{tabular}{l}
		\hline
		\emph{  } Choose step-size $\alpha>0$ and pick any $\bu(0)=\bx(0)\in \R^{n \times p}$ and any $\bd(0)=\by(0)\in \R^{n \times p}$;\\
		\emph{  } Initialize , $\bv(0)=\one\in\R^n$, and $\bV(0)=\dia{\bv(0)}$;\\
		\emph{  } \textbf{for} $k=0,1,\ldots$ \textbf{do}\\
		\qquad$\bv(k+1)=\Wco(k)\bv(k)$; $\bV(k+1)=\dia{\bv(k+1)}$;\\		
		\qquad$\bu(k+1)=\Wco(k)\left(\bu(k)-\alpha\by(k)\right)$;\\
		\qquad$\bx(k+1)=(\bV(k+1))^{-1}\bu(k+1)$;\\
		\qquad$\bd(k+1)=\df(\bx(k+1))$;\\
		\qquad$\by(k+1)=\Wco(k)\by(k)+\bd(k+1)-\bd(k)$;\\
		\emph{  } \textbf{end}\\
		\hline
	\end{tabular}
\end{center}

Since the discussion is under the time-varying set-up, instead of using the right upper corner mark $^k$ as previous, we now use $(k)$ to indicate the $k$-th iteration as well as the time index $k$. In the Push-DIGing algorithm, the aggregated symbols $\bu$, $\bv$, $\bx$, $\by$, $\df(\bx)$ are all defined in a similarly way as we have explained for the quantities $\bx$ and $\sdf(\bx)$ in Section \ref{sec:notation} (each agent $i$ maintains the $i$-th row). The matrix $\Wco$ is a column stochastic matrix (namely, each row sums to $1$) that admits the topology of the network (directed graph). A popular choice of initialization sets $\bd(0)=\by(0)=\df(\bx(0))$. Using uncoordinated step sizes is possible \cite{Nedic2016geometrically} but we will keep using the same step size $\alpha$ across agents for the sake of simplicity. 

To develop an algorithm for resource allocation from Push-DIGing (also considered as a gradient-based method, or termed as explicit method or forward method in different research contexts), we need to first tweak the recursion to produce a proximal method (also termed as implicit method or backward method in different research contexts). The recursive relations of such proximal variant is given by\footnote{A proximal-gradient variant can also be derived following a similar idea.} 
\begin{equation}\label{eq:Push-P-DIGing}
\begin{array}{l}
\bv(k+1)=\Wco(k)\bv(k),
\bV(k+1)=\dia{\bv(k+1)},\\
\bu(k+1)=\Wco(k)\bu(k)-\alpha\by(k+1),\\
\bx(k+1)=(\bV(k+1))^{-1}\bu(k+1),\\
\bd(k+1)=\sdh(\bx(k+1));\\
\by(k+1)=\Wco(k)\by(k)+\bd(k+1)-\bd(k),
\end{array}
\end{equation}
where we always set $\bv(0)=\one\in\R^n$, $\bV(0)=\dia{\bv(0)}$. The remaining quantities should be initialized following the general rules $\bu(0)=\bx(0)\in\R^{n\times p}$ such that $\bh(\bx(0))$ is finite, and $\bd(0)=\by(0)\in\R^{n\times p}$. A specific choice of initialization could be $\bu(0)=\bx(0)=\arg\min_{\bx\in\Omega}\f(\bx)+\frac{1}{2\alpha}\|\bx\|_\Fro^2$ and $\bd(0)=\by(0)=-\frac{1}{\alpha}\bx(0)$. Note that other initialization choices are possible but here we provide a simple one.

The original Push-DIGing algorithm has adopted the so-termed ``Adapt-then-Combine'' (ATC) strategy in its $\bu$-update to accelerate the convergence (see Remark 3 of reference \cite{Nedich2016}). But in the proximal variant \eqref{eq:Push-P-DIGing}, such strategy cannot be applied due to the implementability (see reference \cite{li2017decentralized} for a similar story of not being able to take advantage of the ATC strategy to accelerate the proximal update). Also, we have replaced $\df$ by $\sdh$ to allow a non-smooth objective. Finally the recursive relations \eqref{eq:Push-P-DIGing} with a simple initialization can be resolved as follows.

\begin{center}
	
	\begin{tabular}{l}
		\hline
		\emph{  } Choose step-size $\alpha>0$;\\
		\emph{  } Initialize $\bv(0)=\one\in\R^n$, $\bV(0)=\dia{\bv(0)}$, $\bs(0)=\zero\in\R^n$,\\
		\emph{  } $\bu(0)=\bx(0)=\arg\min_{\bx\in\Omega}\{\f(\bx)+\frac{1}{2\alpha}\|\bx\|_\Fro^2\}$, $\tby(0)=-\bx(0)$;\\
		\emph{  } \textbf{for} $k=0,1,\ldots$ \textbf{do}\\
		\qquad$\bv(k+1)=\Wco(k)\bv(k)$; $\bV(k+1)=\dia{\bv(k+1)}$;\\		
		\qquad$\bs(k+1)=\Wco(k)(\bu(k)-\tby(k))+\bs(k)-\bu(k)$;\\
		\qquad$\bx(k+1)=\arg\min_{\bx\in\Omega}\left\{ \f(\bx)+\frac{1}{2\alpha}\|\bx-(\bV(k+1))^{-1}\bs(k+1)\|_{\bV(k+1)}^2\right\}$;\\
		\qquad$\bu(k+1)=\bV(k+1)\bx(k+1)$;\\
		\qquad$\tby(k+1)=\Wco(k)\bu(k)-\bu(k+1)$;\\
		\emph{  } \textbf{end}\\
		\hline
	\end{tabular}
\end{center}

Having the above intuitive introduction on Push-DIGing and its proximal variation, now we are ready to give the Mirror-Push-DIGing algorithm for decentralized resource allocation over time-varying directed graphs. The design of Mirror-Push-DIGing will start from keeping the recursive relation
\begin{equation}\label{eq:Mirror-Push-DIGing}
\begin{array}{l}
\bv(k+1)=\Wco(k)\bv(k), \bV(k+1)=\dia{\bv(k+1)},\\
\bu(k+1)=\Wco(k)\bu(k)-\alpha\by(k+1),\\
\sdh(\bx(k+1))=(\bV(k+1))^{-1}\bu(k+1),\\
\bd(k+1)=\bx(k+1)-\br;\\
\by(k+1)=\Wco(k)\by(k)+\bd(k+1)-\bd(k),
\end{array}
\end{equation}
where we always set $\bv(0)=\one\in\R^{n\times p}$, $\bV(0)=\dia{\bv(0)}$. The remaining quantities should be initialized following the general rules $\bx(0)\in\Omega$, $\bu(0)=\sdh(\bx(0))$, and $\by(0)=\bd(0)\in\R^{n\times p}$. A specific choice of initialization could be $\bx(0)=\arg\min_{\bx\in\Omega}\{\f(\bx)+\frac{\alpha}{2}\|\bx\|_\Fro^2\}$, $\bu(0)=-\alpha\bx(0)$, and $\by(0)=\bd(0)=\bx(0)-\br$. Again, other initialization choices are possible but we only provide a simple one here. Finally, by resolving relations in \eqref{eq:Mirror-Push-DIGing} and a corresponding simple choice of initial conditions, we finally obtain the implementation of Mirror-Push-DIGing as follows.

\begin{center}
	
	\begin{tabular}{l}
		\hline
		\emph{  } Choose step-size $\alpha>0$ and pick any $\bx(0)\in\Omega$;\\
		\emph{  } Initialize $\bv(0)=\one\in\R^{n}$, $\bV(0)=\dia{\bv(0)}$,\\
		\emph{  } $\bs(0)=\frac{1}{\alpha}\sdf(\bx(0))+\bx(0)$, $\by(0)=\bx(0)-\br$;\\
		\emph{  } \textbf{for} $k=0,1,\ldots$ \textbf{do}\\
		\qquad$\bv(k+1)=\Wco(k)\bv(k)$; $\bV(k+1)=\dia{\bv(k+1)}$;\\		
		\qquad$\bs(k+1)=\Wco(k)\left(\bs(k)-\bx(k)-\by(k)\right)+\bx(k)$;\\
		\qquad$\bx(k+1)=\arg\min_{\bx\in\Omega} \left\{\f(\bx)+\frac{\alpha}{2}\|\bx-\bs(k+1)\|_{(\bV(k+1))^{-1}}^2\right\}$;\\
		\qquad$\by(k+1)=\Wco(k)\by(k)+\bx(k+1)-\bx(k)$;\\
		\emph{  } \textbf{end}\\
		\hline
	\end{tabular}
\end{center}

The per-agent form is listed below in Algorithm 4\footnote{Here in the description of Algorithm 4, $\Niin(k)$ denotes the set of in-neighbors of agent $i$ at iteration/time index $k$. At iteration $k$, an in-neighbor of agent $i$ is defined as a neighboring agent $j$ that can effectively/reliably send its information (quantities $s_j(k)+y_j(k)$, $y_j(k)$, and $v_j(k)$) directly to agent $i$ without relay agents.}.

\newpage
\begin{center}
	{\textbf{Algorithm 4: Mirror-Push-DIGing}}
	
\smallskip	
	\begin{tabular}{l}
		\hline
		\emph{  } Each agent $i$ chooses the same parameter $\alpha>0$ and picks any $x_i(0)\in\Omega_i$;\\
		\emph{  } Each agent $i$ initializes with $v_i(0)=1$, $s_i(0)=\frac{1}{\alpha}\widetilde\nabla f_i(x_i(0))+x_i(0)$, $y_i(0)=x_i(0)-r_i$;\\
		\emph{  } Each agent $i$ \textbf{for} $k=0,1,\ldots$ \textbf{do}\\
		\qquad$v_i(k+1)=\sum_{j\in\Niin(k)\bigcup\{i\}} \wco_{ij}(k)v_j(k);$\\
		\qquad$s_i(k+1)=\sum_{j\in\Niin(k)\bigcup\{i\}} \wco_{ij}(k)\left(s_j(k)-x_j(k)-y_j(k)\right)+x_i(k)$;\\
		\qquad$x_i(k+1)=\arg\min_{x_i\in\Omega_i}\left\{f_i(x_i)+\frac{\alpha}{2v_i(k+1)}\|x_i-s_i(k+1)\|_2^2\right\}$;\\
		\qquad$y_i(k+1)=\sum_{j\in\Niin(k)\bigcup\{i\}} \wco_{ij}(k)y_j(k)+x_i(k+1)-x_i(k)$;\\	
		\emph{  } \textbf{end}\\
		\hline
	\end{tabular}
\end{center}

Here, $\Niin(k)$ denotes the set of in-neighbors of agent $i$ at iteration $k$. 
Each agent $j$ sends the quantities $\wco_{ij}(k)v_j(k)$, $\wco_{ij}(k)\left(s_j(k)-x_j(k)-y_j(k)\right)$
and $\wco_{ij}(k)y_j(k)$ to its out-neighbors. Upon receiving these quantities from its in-neighbors $j\in\Niin(k)$,
each agent $i$ sums these quantities over $j\in\Niin(k)$
and combines them with its own information, at different steps in different ways. 
\cred{We have not performed convergence analysis of Algorithm 4.
Based on what we have obtained in reference~\cite{Nedich2016}, a geometric convergence may be obtained when each $\Omega_i=\R^p$. We have some numerical tests of this algorithm later in Section~\ref{sec:numer}.}

\subsection{Dealing with local couplings}
A generalization of the resource allocation problem which explicitly considers the local linear coupling of multiple resources is as follows
\begin{equation}\label{eq:F_coup}
\begin{array}{cl}
\min\limits_{\bx\in\R^p} & \bh(\bx)=\sum\limits_{i=1}^n h_i(x_i),\\
\st & \sum\limits_{i=1}^n(A_i x_i-r_i)=0.
\end{array}
\end{equation}
Here, $\forall i\in[n]$, we assume $x_i\in\R^{p_i}$, $A_i\in\R^{m\times p_i}$, $r_i\in\R^m$, and obviously $p=\sum_{i=1}^n p_i$. To simplify our presentation, from now on, we will use
\begin{equation}
\bx\triangleq\left(
\begin{array}{c}
x_1\\
x_2\\
\vdots \\
x_n
\end{array}
\right)\in\R^p;\quad
\bA\triangleq\left(
\begin{array}{cccc}
A_1&0&\cdots&0\\
0&A_2&\cdots&0\\
\vdots&\vdots&\ddots&\vdots\\
0&0&\cdots&A_n\\
\end{array}
\right)\in\R^{(mn)\times p};\quad
\br\triangleq\left(
\begin{array}{c}
r_1\\
r_2\\
\vdots \\
r_n
\end{array}
\right)\in\R^{mn};
\end{equation} 
Note that, unlike previous sections, we are no longer using the aggregated/compact notation in matrix forms since it becomes inconvenient when $x_i$'s have different dimensions. But still as that in previous sections, the function $h_i\triangleq f_i+g_i$ is only assumed to be convex and already include the indicator function of the local constraint $x_i\in\Omega_i$ in itself. 

The optimality condition of \eqref{eq:F_coup} is
\begin{subequations}\label{eq:F_coup_opt}
	\begin{numcases}{}
	\sdhi_i(x_i^*)=A_i^\T \xi_i^*, \forall i,\label{eq:F_coup_opt1}\\
	\xi_1^*=\xi_2^*=\ldots=\xi_n^*,\label{eq:F_coup_opt2}\\
	\sum\limits_{i=1}^n(A_i x_i^*-r_i)=0,\label{eq:F_coup_opt3}
	\end{numcases}
\end{subequations}
where we have introduced a group of auxiliary variables $\xi_i\in\R^m,\ \forall i\in[n]$. Again, here we can spot the consensus relation \eqref{eq:F_coup_opt2} and the ``summing to zero'' relation \eqref{eq:F_coup_opt3}, similar forms of which have appeared in the optimality condition of the original resource allocation problem. Now if the condition \eqref{eq:F_coup_opt1} that bridges $x_i$'s and the auxiliary variables $\xi_i$'s is further met, then $\{x_i^*\}_{i=1}^n$ is an optimal solution of \eqref{eq:F_coup}. Let us denote $\sdh(\bx)\triangleq\left(\sdhi_1(x_1);\sdhi_2(x_2);\cdots;\sdhi_n(x_n)\right)\in\R^p$, $\bxi\triangleq\left(\xi_1;\xi_2;\cdots;\xi_n\right)\in\R^{mn}$, and $\eLap\triangleq\Lap\otimes I_m\in\R^{(mn)\times (mn)}$ where $I_m$ is an $m$-by-$m$ identity matrix and ``$\otimes$'' represents the Kronecker product. This way, \eqref{eq:F_coup_opt} is compactly rewritten as
\begin{subequations}\label{eq:F_coup_opt_comp}
	\begin{numcases}{}
	\sdh(\bx^*)=\bA^\T \bxi^*,\label{eq:F_coup_opt_comp1}\\
	\eLap\bxi^*=\zero,\label{eq:F_coup_opt_comp2}\\
	(\one\otimes I_m)^\T(\bA\bx^*-\br)=0,\label{eq:F_coup_opt_comp3}
	\end{numcases}
\end{subequations}

To eventually converge to a point that meets \eqref{eq:F_coup_opt_comp}, we can construct sequences that satisfy the following recursive relations:
\begin{subequations}\label{eq:updates_coup}
	\begin{align}
	&\bA\bx^{k+1}-\br+c\by^{k+1}+(\bB-c\eLap)(\bxi^{k+1}-\bxi^k)=\zero,\label{eq:up_coup_line1}\\
	&\sdh(\bx^{k+1})=\bA^\T\bxi^{k+1},\label{eq:up_coup_line2}\\
	&\sdh(\bx^k)=\bA^\T\bxi^k,\label{eq:up_coup_line3}\\
	&\by^{k+1}=\by^k+\eLap\bxi^{k+1},\label{eq:up_coup_line4}
	\end{align}
\end{subequations}
where $\bB=\text{blkdiag}\{B_1,B_2,\ldots,B_n\}\in\R^{(mn)\times(mn)}$ is a block diagonal matrix that contains all the local step sizes $B_i$'s. Each matrix step size $B_i$ is maintained by agent $i$ and $\bB-c\eLap\succcurlyeq0$ should be fulfilled. One can see that \eqref{eq:updates_coup} is nothing but mimicking \eqref{eq:updates2}, the recursive relations of Mirror-P-EXTRA (Algorithm 1), while preserving $\sdh(\bx^k)=\bA^\T\bxi^k$. One can verify that if  \eqref{eq:updates_coup} converges, its sequences $(\bx^k,\bxi^k)$ converge to $(\bx^*,\bxi^*)$. Finally recursion \eqref{eq:updates_coup} can be resolved as the following implementation
\begin{subequations}\label{eq:updates_coup_comp}
	\begin{align}
	&\bx^{k+1}=\arg\min_{\bx\in\Omega}\left\{\f(\bx)+\|\bx\|_{\bA^\T\bB^{-1}\bA}^2+\langle-\bA^\T\bB^{-1}\br-\bA^\T\bxi^k+2c\bA^\T\bB^{-1}(2\by^k-\by^{k-1}),\bx\rangle\right\},\label{eq:up_coup_comp_line1}\\
	&\bxi^{k+1}=\bxi^k-c\bB^{-1}(2\by^k-\by^{k-1})-\bB^{-1}(\bA\bx^{k+1}-\br),\\
	&\by^{k+1}=\by^k+\eLap\bxi^{k+1},\label{eq:up_coup_comp_line4}
	\end{align}
\end{subequations}
starting from $\bxi^0\in\R^p$, $\bx^0=\arg\min_{\bx\in\Omega}\f(\bx)+\langle\bA^\T\bxi^0,\bx\rangle$, $\by^{-1}=0$, and $\by^0=\eLap\bxi^0$.
The per-agent update form is listed below in Algorithm 5.

\smallskip
\begin{center}
	{\textbf{Algorithm 5: Mirror-P-EXTRA handling local couplings}}
	
\smallskip	
	\begin{tabular}{l}
		\hline
		\emph{  } Each agent $i$ chooses its own parameter matrix $B_i\in\R^{m\times m}$ and the same parameter $c>0$;\\
		\emph{  } Each agent $i$ initializes with arbitrary $\xi_i^0\in\R^{p_i}$ and sets $x_i^0=\arg\min_{x_i\in\Omega_i} f_i(x_i)+\langle A_i^\T\xi_i^0,x_i\rangle$;\\
		\emph{  } Each agent $i$ initializes with $y_i^{-1}=0$;\\
		\emph{  } Each agent $i$ \textbf{for} $k=0,1,\ldots$ \textbf{do}\\
		\qquad$y_i^{k}=y_i^{k-1}+\sum_{j\in\Ni\bigcup\{i\}}\L_{ij}\xi_j^k$\\
		\qquad$x_i^{k+1}=\arg\min_{x_i\in\Omega_i} \left\{f_i(x_i)+\|x_i\|_{A_i^\T B_i^{-1} A_i}^2+\langle-A_i^\T B_i^{-1}r_i-A_i^\T\xi_i^k+2cA_i^\T B_i^{-1}(2y_i^k-y_i^{k-1}),x_i\rangle\right\}$;\\
		\qquad$\xi_i^{k+1}=\xi_i^k-cB_i^{-1}\left(2y_i^k-y_i^{k-1}\right)-B_i^{-1}(A_ix_i^{k+1}-r_i)$;\\
		\emph{  } \textbf{end}\\
		\hline
	\end{tabular}
\end{center}
It can be seen that this algorithm degenerates to Mirror-P-EXTRA (Algorithm 1) when $A_i=I$ for all $i$. 
\cred{We have not analyzed the convergence properties of this algorithm. We believe that its convergence behavior is similar to that of Mirror-P-EXTRA, since it is a generalization of Mirror-P-EXTRA with an intuitive modification introduced above.}

\section{Numerical Experiments}\label{sec:numer}
The experiments are conducted over a fixed undirected connected graph with 
$n=100$ vertices and $|\E|=198$ edges. 
The ``connectivity ratio'' of the undirected graph is defined by 
$r_\mathrm{c}=|\E|/(0.5n(n-1))=0.04$ (the average degree is $3.96$), where $0.5n(n-1)$ is the maximum 
number of edges an undirected graph can have. We generate the graphs randomly\footnote{To guarantee connectedness of the graph, we first grow a random tree; then uniformly randomly add edges into the graph to reach the specified connectivity ratio.}. 
The objective function $f_i$ of agent $i$ is given by
\[f_i(x_i)=0.5x_i^\T H_i^\T H_i x_i+b_ix_i,\]
where $H_i\in\R^{2\times2}$ and $b_i\in\R^{2\times2}$ have the entries 
generated by the normal distribution with zero mean and unit variance. 
For each agent $i$, the local constraint set is a box,
\[\Omega_i=\{(\omega_{i,1},\omega_{i,2})\in\R^2\big|
0\leq\omega_{i,1}\leq\overline{\omega}_{i,1},\ 
0\leq\omega_{i,2}\leq\overline{\omega}_{i,2}\},\] 
where the interval boundaries $\overline{\omega}_{i,j}$ are 
randomly generated following the uniform distribution over the interval 
$[1,2]$ for all $i\in[n]$ and $j=1,2$.

For each agent $i\in[n]$, the resource vector $r_i$ is the mean value of the interval constraints, i.e.,
$r_i=[\overline{\omega}_{i,1}/2;\overline{\omega}_{i,2}/2]\in\R^2$. 
If the optimal solution of the problem does not hit the boundary of the set $\Omg$, 
multiple trials are made to obtain the problem for which the constraint set $\Omg$ is active at the optimal solution.
Mirror-EXTRA cannot be applied to solve such problems, so we implement Mirror-PG-EXTRA that uses a gradient-projection update.

\an{In the experiments, we use the matrix $\Lap=0.5(I-W)$ 
	for both Mirror-P-EXTRA and Mirror-PG-EXTRA, 
	where $W$ is generated according to the Metropolis-Hasting rule.
	For Mirror-P-EXTRA we use $c=0.01/(\mu L\lams{\Lap})^{0.5}$, which is based on 
	Corollary~\ref{corollary:scalability}. 
	The constant factor $0.01$ is hand-tuned and found to be effective 
	for most of our randomly generated graphs 
	and randomly generated 
	$\f$, $\br$, and $\Omg$. To verify the viability of using different parameters $\beta_i$ across agents, 
	we choose $\beta_i=\phi_i c\laml{\Lap}$ for each $i$, 
	where $\phi_i$ is a random variable following the uniform distribution over the interval 
	$[1,1.5]$. For Mirror-PG-EXTRA we set $c=0.5/L$ and $\beta_i=\phi_i c$ for each agent $i$, 
	where $\phi_i$ is a random variable following the uniform distribution over the interval $[1,1.5]$.} 

To compare the competitiveness of the algorithms of Section \ref{sec:algos} with those in the existing literature, 
we implement the DPDA-S algorithm of~\cite{Aybat2016distributed}. DPDA-S has one system-level parameter 
$\gamma$ and two per-agent parameters, $\tau_i$ and $\kappa_i$. 
Based on the recommended parametric structure as given in Remark II.1 of~\cite{Aybat2016distributed},
which sets $\tau_i$ and $\kappa_i$ automatically to produce another parameter 
$c_i$, we tuned the parameter $c_i$ and $\gamma$ to obtain one plot for this algorithm. 
In another plot for this algorithm, we hand-optimized all the parameters to achieve a better performance. 
In Mirror-P-EXTRA, there is a system-level parameter $c$ 
which can be set based on the per-agent parameters $\beta_i$. 
Compared to Mirror-EXTRA, DPDA-S requires a finer tune in its parameters to achieve a competitive performance. The convergence curves are shown in Fig.~\ref{eps:ED_fig1}.

\begin{figure}[H]
	\begin{center}
		\vspace{-1em}
		\includegraphics
		[width=0.6\linewidth]{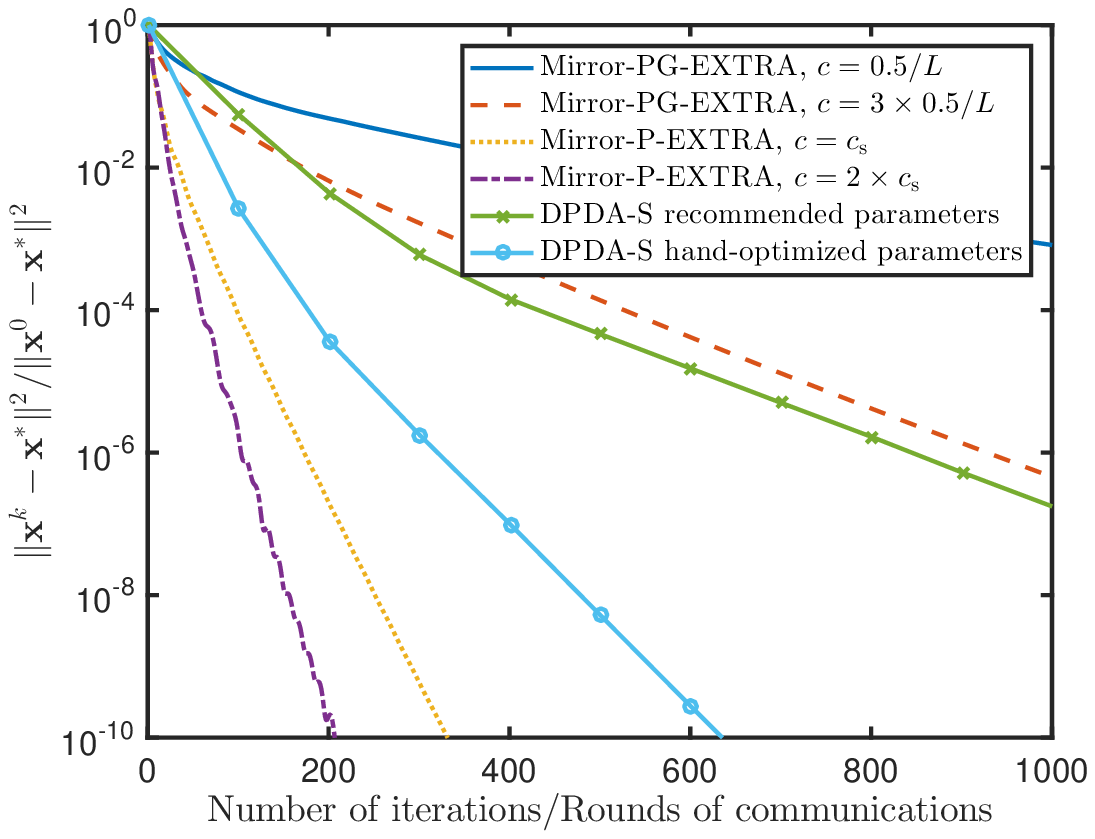}\vspace{-1em}
		\caption{Plots of the normalized residual $\frac{\|\bx^k-\bx^*\|_\mathrm{F}}{\|\bx^0-\bx^*\|_\mathrm{F}}$. The step size for Mirror-P-EXTRA, $c_\mathrm{s}=0.01/(\mu L\lams{\Lap})^{0.5}$, is based on Corollary \ref{corollary:scalability}. The constant $0.01$ in the numerator is hand-tuned and found to be effective 
			for most of our randomly generated graphs ($r_\mathrm{c}=4/n=4/100$) and randomly generated $\f$, $\br$, and $\Omg$. For Mirror-PG-EXTRA, a step size larger than $3\times0.5/L$ will lead to divergence in the current trial. For Mirror-P-EXTRA, the parameter $2\times c_s$ gives the fastest convergence speed in the current trial.\label{eps:ED_fig1}\vspace{-3em}}
	\end{center}
\end{figure}

\cred{In the above numerical test (see Fig. \ref{eps:ED_fig1}), the outcome of Mirror-P-EXTRA outperforming the other algorithms in the number of iterations is in expectation. Mirror-P-EXTRA has to pay extra computational effort (solving a constrained convex optimization problem) at each iteration compared to other competitive algorithms. However, in decentralized computing, it may worth it to perform more computations before exchanging information in the following round of communication because communication costs and delays are usually considered more significant compared to that caused by local computations. There is actually a trade-off between the total computational cost/time and the total communication cost/delay. Which algorithm is more preferable depends on the hardness of the specific optimization problem and the performance of the underlying cyber system.}

\begin{figure}[H]
	\begin{center}
		\vspace{-1em}
		\includegraphics
		[width=0.6\linewidth]{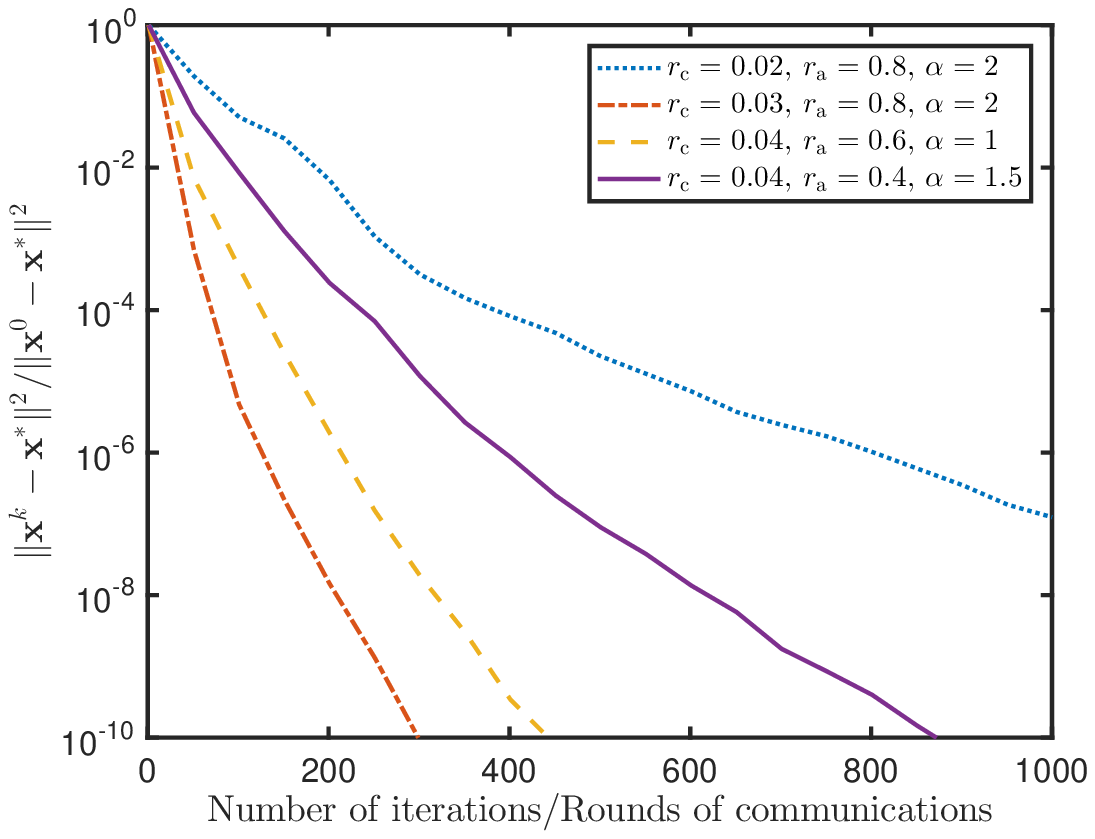}\vspace{-1em}
		\caption{Plots of the normalized residual $\frac{\|\bx^k-\bx^*\|_\mathrm{F}}{\|\bx^0-\bx^*\|_\mathrm{F}}$. The step sizes for Mirror-Push-DIGing are roughly hand tuned. $r_\mathrm{c}$ is the connectivity ratio and $r_\mathrm{a}$ is the activation ratio.\label{eps:ED_fig2}\vspace{-3em}}
	\end{center}
\end{figure}

To show the viability of the Mirror-Push-DIGing algorithm (see Section \ref{sec:ext}) for time-varying directed graphs, we randomly
randomly generate a directed graph with $n=100$ vertices. 
Such a graph can have at most $n(n-1)=9900$ directed links. 
We define the connectivity ratio of a directed graph on $n$ vertices by 
$r_\mathrm{c}=|\A|/n(n-1)$ where $\A$ is the arc set. For a time-varying graph sequence $\GTVdir(k)=\{\V,\A(k)\}$, the activation ratio at time $k$ is defined as $r_\mathrm{a}(k)=|\A(k)|/|\A|$, where $\A$ is the set of links of the underlying(fixed) digraph $\A=\bigcup_k\A(k)$. 
Thus, $r_\mathrm{a}(k)\times r_\mathrm{c}$ is equal to the connectivity ratio of $\GTVdir(k)$ at time $k$. When $r_\mathrm{a}(k)$ is a constant for all $k$, we drop the time index and simply use $r_\mathrm{a}$. In the experiment, we run Mirror-Push-DIGing over a few different time-varying directed graph sequences. 
The graph sequences are generated by uniformly randomly activating the arcs of the above mentioned underlying digraph. The results are plotted in Fig.~\ref{eps:ED_fig2}. 
More details of settings for $r_\mathrm{c}$, $r_\mathrm{a}$, and the step sizes $\alpha$ are given in the legend of Fig~\ref{eps:ED_fig2}. The step sizes for Mirror-Push-DIGing are roughly hand tuned to obtain relatively fast convergence.

\section{Conclusion}\label{sec:concl}
\an{In this paper, 
	we have presented an interesting relationship 
	between resource allocation problem and the consensus optimization problem. 
	Based on this relation,
	we have proposed two algorithms, namely Mirror-P-EXTRA and Mirror-EXTRA,
	for distributed resource allocation in a static connected undirected graph.
	We have established the convergence and convergence rate properties
	of the algorithms. In particular, we have shown that both of the algorithms enjoy 
	an $R$-linear convergence rate 
	when the resource allocation problem has strongly convex objective function with
	Lipschitz continuous gradients and does not have additional set constraints.
	We also have illustrated the convergence behavior of the Mirror-P-EXTRA and its computationally less 
	expensive projection-based variant (Mirror-PG-EXTRA) by some numerical experiments.}

\bibliographystyle{IEEEtran}
\bibliography{document.bib}
\end{document}